\documentclass[11pt]{article}
\usepackage[margin=1in]{geometry}
\usepackage{microtype}
\usepackage[numbers]{natbib}

\usepackage{amssymb}
\usepackage{mathtools}
\usepackage{nicefrac}
\usepackage{amsthm}

\usepackage[theorems,skins]{tcolorbox}
\usepackage{graphicx}
\usepackage{subfigure}
\usepackage{booktabs}
\usepackage{xcolor}
\usepackage[hypertexnames=false]{hyperref}
\usepackage[capitalize,noabbrev]{cleveref}

\theoremstyle{plain}
\newtheorem{theorem}{Theorem}[section]
\newtheorem{proposition}[theorem]{Proposition}
\newtheorem{lemma}[theorem]{Lemma}

\theoremstyle{definition}
\newtheorem{definition}[theorem]{Definition}
\newtheorem{assumption}[theorem]{Assumption}

\theoremstyle{remark}

\hypersetup{
	colorlinks   = true,               %
	urlcolor     = blue,               %
	linkcolor    = {blue!50!black},    %
	citecolor   = {blue!50!black}      %
}
\usepackage{url}
\usepackage{float}
\usepackage{makecell}
\usepackage[flushleft]{threeparttable}
\usepackage{colortbl} %
\usepackage{cellspace}
\usepackage{multirow}

\usepackage{tikz,pgfplots}
\pgfplotsset{compat=newest}
\usetikzlibrary{calc}      %
\usetikzlibrary{arrows.meta}

\newcommand*\colourcheck[1]{%
	\expandafter\newcommand\csname gcmark\endcsname{\textcolor{#1}{\ding{52}}}%
}
\newcommand*\colourxmark[1]{%
	\expandafter\newcommand\csname rxmark\endcsname{\textcolor{#1}{\ding{55}}}%
}
\definecolor{mygreen}{HTML}{02862a}
\definecolor{myred}{HTML}{9a0000}
\colourcheck{mygreen}
\colourxmark{myred}
\definecolor{linen}{HTML}{FAF0E6}
\definecolor{darkteal}{RGB}{0, 110, 110}

\usepackage{preamble}       %
\usepackage[opt]{commands}  %

\newcommand{\CondExp}[2]{{\mathbb{E}}\left[#1 \,\middle| \, #2\right]}

\newcommand{\Nb}{\mathbb{N}}

\usepackage{xspace}
\newcommand{\algname}[1]{{\sf #1}\xspace}

\usepackage{authblk}

\title{\textbf{Can SGD Handle Heavy-Tailed Noise?}%
\thanks{This manuscript is an extended version of a paper accepted for publication in the \emph{SIAM Journal on Optimization}.}}

\author[1,2]{Ilyas Fatkhullin%
\thanks{Corresponding author: \texttt{{\color{black}{fatkhullin.ilyas@gmail.com}\allowbreak}}}%
\thanks{Part of this work was conducted while I.F. was a research visitor at the Georgia Institute of Technology.}}
\author[1]{Florian Hübler}
\author[2]{Guanghui Lan}

\affil[1]{ETH Zurich, Zurich, Switzerland}
\affil[2]{Georgia Institute of Technology, Atlanta, Georgia, USA}
\date{September 24, 2026}

\begin{document}

\maketitle

\begin{abstract}
    Stochastic Gradient Descent (SGD) is a cornerstone of large-scale optimization, yet its theoretical behavior under heavy-tailed noise---common in modern machine learning and reinforcement learning---remains poorly understood. In this work, we rigorously investigate whether vanilla SGD, devoid of any adaptive modifications, can provably succeed under such adverse stochastic conditions. Assuming only that stochastic gradients have bounded $p$-th moments for some $p \in (1, 2]$, we establish sharp convergence guarantees for (projected) SGD across convex, strongly convex, and non-convex problem classes.
    In particular, we show that SGD achieves minimax optimal sample complexity under minimal assumptions in the convex and strongly convex regimes: $\mathcal{O}(\varepsilon^{-\frac{p}{p-1}})$ and $\mathcal{O}(\varepsilon^{-\frac{p}{2(p-1)}})$, respectively.
    {\color{black}{For non-convex objectives, under standard smoothness and a bounded central $p$-th moment, we prove convergence to a stationary point with $\mathcal{O}(\varepsilon^{-\frac{2p}{p-1}})$ complexity, and complement this with a matching {SGD}-specific lower bound, showing that this rate cannot be improved by any deterministic positive non-increasing step-size schedule.}\allowbreak} These results challenge the prevailing view that heavy-tailed noise renders SGD ineffective, and establish vanilla SGD as a robust and theoretically principled baseline---even in regimes where the variance is unbounded.
\end{abstract}



\section{Introduction}\label{sec:intro}
Consider a stochastic optimization problem
\begin{eqnarray}\label{eq:problem}
	\min_{x \in \cX} F(x) := \Exp{f(x, \xi)} ,
\end{eqnarray}
where $\cX \subseteq \R^d$ is closed and convex, and $\xi$ is a random variable distributed according to some unknown distribution $\cD.$ Access to $F$ is only available through unbiased stochastic gradients $\nabla f(x, \xi)$, which may exhibit heavy-tailed behavior. Recent empirical studies have highlighted the prevalence of heavy-tailed phenomena in modern machine learning datasets and environments, especially in deep learning \citep{simsekli2019tail,battash2023revisiting,ahn2023linear} and reinforcement learning \citep{Garg_PPO_heavy_tail_2021}. To capture heavy-tailedness formally, we adopt the following moment assumptions.

\begin{assumption}\label{assum:pBM}
	Let $F(\cdot)$ be differentiable on $\cX$.\footnote{In the convex case, we can lift differentiability assumption and work with sub-differentials, but we chose to work with differentiable setting to simplify the exposition and unify assumptions with non-convex setting.} We have access to stochastic gradients with $\Exp{\nabla f(x, \xi)} = \nabla F(x)$ and there exists $p \in (1, 2]$ such that the $p$-th moment is bounded, i.e.,
	$$
	p\text{-BM} \qquad \Exp{\norm{\nabla f(x, \xi)}^p} \leq G^p \qquad \text{for all } x \in \cX .
	$$
\end{assumption}
Under an additional smoothness condition, we consider a refined model referred to as $p$-th Bounded Central Moment or $p$-BCM for short: $\Exp{\norm{\nabla f(x, \xi) - \nabla F(x)}^p} \leq \sigma^p \quad \text{for all } x \in \cX . $\footnote{{\color{black}{The two moment assumptions play different role in this work. The $p$-BM assumption controls the full stochastic gradient and is used in the convex, strongly convex, and Hölder-smooth non-convex analyses. The \mbox{\hyperref[assum:pBCM]{(p-BCM)}} assumption controls only the gradient noise and is used in the standard-smooth non-convex analysis; see \mbox{\Cref{subsec:mini_batch_SGD}}.}\allowbreak}} These assumptions are standard in recent theoretical works on heavy-tailed optimization. When $p<2$, gradient estimates can exhibit unbounded variance, precluding the use of conventional analysis techniques. Such heavy-tailed behavior of data is often used to explain the empirical success of adaptive algorithms over vanilla \algname{SGD}. The examples of such adaptive schemes include methods based on gradient clipping \citep{zhang2020adaptive,sadiev2023high,nguyen2023improved}, Normalized-SGD \citep{hubler2024gradient,liu2024nonconvex,he2025complexity}, their combinations \citep{cutkosky2021high,chezhegov2024clipping}, and even more general non-linear schemes \citep{polyak1979adaptive,jakovetic2023nonlinear,NonlinearHP2023Armacki}. In this work, we revisit the analysis of vanilla Stochastic Gradient Descent with (optional) projection under heavy-tailed noise
\begin{tcolorbox}[colback=white,colframe=mygreen!75!black]
  \setlength\abovedisplayskip{0pt}
  \setlength\belowdisplayskip{0pt}
	\begin{align}
		\text{\algname{SGD}:} \qquad \text{step-size sequence, } \cb{\eta_t}_{t\geq 1} , \notag \qquad
		x_{t+1} = \Pi_{\cX}(x_t - \eta_t \nabla f(x_{t}, \xi_{t}) ) ,
	\end{align}
\end{tcolorbox}
\noindent where $\Pi_{\cX}(\cdot)$ is the Euclidean projection onto $\cX.$\footnote{The projection is optional for most results but it is useful to discuss the implications when $\cX$ is bounded.} While this is one of the simplest and most popular stochastic optimization algorithms, its convergence in the heavy-tailed settings remains elusive. While many adaptive methods mentioned above can achieve optimal convergence, the majority of these works hardly question if \algname{SGD} may have similar properties as these more sophisticated adaptive algorithms. The main argument discussed in the literature concerning the failure of \algname{SGD} in such settings is centered around the following simple example. Consider the $1$-dimensional quadratic function $F(x) = \frac{1}{2} x^2$ on $\cX = \R$ and let $\xi \sim \cD$ be any zero-mean noise with infinite variance and bounded $p$-th moment.\footnote{E.g., we can use two-sided Pareto distribution with tail index $\alpha = 2.$} In that case we have after one step of \algname{SGD} starting from $x_1 = 0$
\begin{equation}\label{eq:quadratic_example}
\Exp{F(x_2)} = \frac{1}{2} \Exp{ \sqnorm{\nabla F(x_2)} } = \frac{1}{2} \Exp{\eta_1^2 \sqnorm{\xi_1}} = + \infty .
\end{equation}
The last equality holds due to infinite variance whenever the step-size $\eta_1$ is non-adaptive, i.e., predefined/deterministic. Therefore, vanilla \algname{SGD} does not converge in the usual sense neither in expectation of the function value nor gradient norm squared. One might make an erroneous conclusion out of this example that \algname{SGD} is useless under this heavy-tailed model of noise if $p < 2$. However, this example is fairly artificial and in this work we aim to investigate under what conditions and in what sense \algname{SGD} converges under infinite variance.

\begin{table}[ht]
    \centering
    \begin{threeparttable}
    \begin{tabular}{m{2.5cm}m{3.5cm}m{4.0cm}m{3.5cm}}
    \hline
        & \centering\textbf{Convex} & \centering\textbf{Strongly Convex} & \centering\textbf{Non-convex}
    \tabularnewline \hline
        \makecell{Converg. \\ Criterion} &
        \centering $\mathbb{E}[F(\widetilde x_T) - F^*] \leq \varepsilon$ &
        \centering\small \makecell{$\mathbb{E}[(F(\bar x_T) - F^*)^{\nicefrac{p}{2}}] \leq \varepsilon^{\nicefrac{p}{2}}$\\or $\mathbb E[\|x_T - x^*\|^p] \leq \varepsilon^{\nicefrac{p}{2}}$ }&
        \centering $\mathbb{E}[\|\nabla F(\bar x_T)\|^2] \leq \varepsilon^2$
    \tabularnewline \hline
        \rowcolor{linen}
        \makecell{Step-size \\ Order} &
        \centering\makecell{ $\eta_t = \frac{D_{\cX}}{G \, t^{\nicefrac{1}{p}}}$  } &
        \centering\makecell{ $\eta_t = \frac{2}{\mu \, t} $ } &
        \centering\makecell{ $\eta_t = \rb{\frac{\Delta_1}{L_p G^p} \frac{1}{t}}^{\nicefrac{1}{p}} $ }
    \tabularnewline
        \rowcolor{linen}
        \makecell{ Complexity \\ in $\Theta(\cdot)$ } &
        \centering\makecell{
            $\rb{\frac{G D_{\cX}}{ \varepsilon}}^{\frac{p}{p - 1}}$ \\ \\
            Thm.~\ref{thm:convex_SGD}, \citep{nemirovskij_yudin_1979_eff}
        } &
        \centering\makecell{
            $\rb{\frac{G^2}{\mu \, \varepsilon}}^{\frac{p}{2(p - 1)}}$ \\  \\
            Thm.~\ref{thm:SCprojectedSGD_improved}, \citep{zhang2020adaptive}
        } &
        \centering\makecell{
            $ \Delta_1 \rb{\frac{L_p^{\nicefrac{1}{p} } G}{\varepsilon^2}}^{\frac{p}{p-1}} $ %
            \\ \\ {\renewcommand\crefpairconjunction{, }
            Thms.~\ref{thm:NC_SGD_upper}, \ref{thm:nonconvex.lower_bound.SGD_lower_bound}}
        }
    \tabularnewline \hline
    \end{tabular}
     \caption{
    Summary of sample/iteration complexity bounds of \algname{SGD} for solving \eqref{eq:problem} under \hyperref[assum:pBM]{(p-BM)} \Cref{assum:pBM}. The ``Converg. Criterion'' row specifies the convergence criterion/measure for each setting with $\widetilde x_T$, $\bar x_T$, $x_T$ denoting \textit{average}, \textit{random} and  \textit{last} iterate convergence respectively. ``Step-size Order'' column reports the step-size order to achieve the complexity in the row ``Complexity in  $\Theta(\cdot)$''. The symbol $\Theta(\cdot)$ means that our rates are unimprovable for \algname{SGD} under our assumptions. In convex and strongly-convex cases our complexities are minmax optimal for any first-order algorithms; matching lower bounds are established in \citep{nemirovskij_yudin_1979_eff}, \citep{zhang2020adaptive}. In the non-convex setting, we derive an algorithm-specific lower bound to show tightness of our rates in all problem parameters. The constants $D_{\cX}$, $\mu$, $\Delta_1$ and $L_p$ denote the diameter of $\cX$, strong convexity modulus, initial function value gap and the Hölder smoothness constant respectively; refer to corresponding section for formal definitions.
    }
    \label{tab:summary}
    \end{threeparttable}
\end{table}

\paragraph{Contributions}
{\color{black}{We study \algname{SGD} under \mbox{\hyperref[assum:pBM]{(p-BM)}} across the
convex, strongly convex, and Hölder-smooth non-convex settings, and under
\mbox{\hyperref[assum:pBCM]{(p-BCM)}} in the standard-smooth non-convex setting. The
$p$-BM results are summarized in \Cref{tab:summary}.}\allowbreak} The key contributions are as follows:

 \textbf{Convex.} We establish that a weighted average function value of \algname{SGD} converges for a wide range of step-size sequences. Moreover, if the diameter of the set $\cX$ is bounded, even the simple average iterate of \algname{SGD} converges in expectation. In the latter case, when the step-size sequence $\eta_t$ is tuned properly, \algname{SGD} achieves the optimal sample complexity $\cO(\varepsilon^{-\frac{p}{p-1}})$ to find $\widetilde x_T$ with $\mathbb{E}[F(\widetilde x_T) - F^*] \leq \varepsilon.$\footnote{Throughout the work, we use the standard $\Oc\pare{\cdot}, \Omega\pare{\cdot}, \Theta\pare{\cdot}$ complexity notations \citep{MultivariateO}, $\Oct\pare{\cdot}$ additionally hides poly-logarithmic factors. In some cases, we slightly abuse this notation to highlight the dependence of the complexities on certain variables of our focus, e.g., $\varepsilon$, $p$, which will hopefully be clear from the context.} {\color{black}{We also complement this upper bound with a nearly tight high-probability lower bound for a large class of non-adaptive algorithms, including \algname{SGD} with predefined step-sizes and standard output rules. In particular, for $p\in(1,2)$ and any $\alpha \in (p,2)$, guaranteeing $F(x_{T}^{\text{out}}) - F^* \leq \varepsilon$ with failure probability at most $\delta$ requires at least $\Omega\rb{\varepsilon^{-\frac{\alpha}{\alpha-1}} \delta^{-\frac{1}{\alpha-1}}}$ samples.}\allowbreak}

\textbf{Strongly convex.} In this case we show that \algname{SGD} with step-size $\eta_t = {2}/{\color{black}{(}}{\mu \, t}{\color{black}{)}}$, $t \geq 1$ converges in function value, $\mathbb{E}[\rb{F(\bar x_T) - F^*}^{\nicefrac{p}{2}}] \leq \varepsilon^{\nicefrac{p}{2}} $ for a point $\bar x_T$ sampled uniformly from the iterates $\cb{x_t}_{t\leq T},$ and in terms of the distance to the optimum, $\mathbb E[\|x_T - x^*\|^p] \leq \varepsilon^{\nicefrac{p}{2}},$ with optimal sample complexity $\cO(\varepsilon^{-\frac{p}{2(p-1)}})$. This implies that in strongly convex case, \algname{SGD} only requires knowledge of strong convexity modulus $\mu$ and is the first optimal algorithm that does not require knowledge of the tail index $p$.

 \textbf{Non-convex {\color{black}{under Hölder smoothness.}}} If function $F(\cdot)$ is Hölder smooth of order $\nu = p$,\footnote{It is important for our analysis that Hölder smoothness order $\nu$ does not exceed the order of the tail index $p.$ We consider the case $\nu = p$ for simplicity.} we show that \algname{SGD} converges to a stationary point in expectation. In particular, in the unconstrained case this implies that we can find a point $\bar x_T$ with $\mathbb{E}\norm{\nabla F(\bar x_T)}^2 \leq \varepsilon^2$ after $\cO\rb{L_p^{\frac{1}{p-1}}\varepsilon^{-\frac{2p}{p-1}}}$ iterations of \algname{SGD}, where $L_p$ is the Hölder constant. Furthermore, we show tightness of this upper bound by constructing a matching algorithm-specific lower bound for \algname{SGD} with arbitrary polynomial step-size sequence. \textcolor{black}{This lower bound construction is novel even for $p=2$ and can be of independent interest.}

\begin{sloppypar}
{\color{black}{\textbf{Non-convex under standard smoothness.} Finally, under standard $L$-smoothness and $p$-BCM, we prove that, in the noise-dominated regime, \algname{SGD} has sample complexity
$\cO\rb{\pare{\sigma\sqrt{L\Delta_1}/\varepsilon^2}^{\frac{p}{p-1}}}$
for finding a point $\bar x_T$ with
$\mathbb{E}\norm{\nabla F(\bar x_T)}^p \leq \varepsilon^p$. We complement
this guarantee with a matching \algname{SGD}-specific lower bound at fixed
confidence for every deterministic positive non-increasing step-size schedule.
For $p<2$, this proves that vanilla \algname{SGD} is not minimax optimal:
adaptive clipping and normalization improve the exponent in $1/\varepsilon$
from $2p/(p-1)$ to $2+p/(p-1)$. Moreover, our lower bound for \algname{SGD}
contains a multiplicative factor $\Omega\rb{\delta^{-\frac{1}{p-1}}}$, where $\delta$ is a failure probability, whereas normalized
methods in prior work incur only a polylogarithmic factor in $1/\delta$.}\allowbreak}
\end{sloppypar}

{\color{black}{Overall, our results characterize both the capabilities and the limitations of
vanilla \algname{SGD} under heavy-tailed noise. It is minimax optimal in the
convex and strongly convex regimes, while our non-convex lower bounds identify
when adaptive methods are necessary for optimal accuracy and confidence
dependence.
}\allowbreak}

\section{Related Work}
In this section, we will briefly summarize the literature related to optimization under heavy-tailed noise, and will make more detailed comparison to most related work in the discussion in the subsequent sections.

\textbf{Adaptive methods under heavy-tailed noise.}
A growing body of work develops adaptive algorithms that achieve provable convergence under heavy-tailed noise. These include mirror descent variants \citep{nemirovskij_yudin_1979_eff,vural2022mirror}, \algname{Clip-SGD} \citep{zhang2020adaptive,sadiev2023high,nguyen2023improved,liu2023stochastic}, \algname{Normalized SGD} \citep{hubler2024gradient}, \algname{Clip-AdaGrad} \citep{chezhegov2024clipping}, \algname{Clip-SGD} with normalization and momentum \citep{cutkosky2021high} and other non-linear schemes \citep{polyak1979adaptive,jakovetic2023nonlinear,NonlinearHP2023Armacki}. Some of these works go beyond in-expectation analysis and develop high-probability guarantees with polylogarithmic dependence on the inverse failure probability, $1/\delta$, thanks to adaptive step-sizes or other complex techniques such as robust distance estimation \citep{davis2021low} and robust gradient aggregation methods \citep{puchkin2024breaking}. Although adaptive methods often offer strong convergence guarantees, vanilla \algname{SGD} may retain distinct advantages. For instance, \cite{wan2023implicit} show that \algname{SGD} under heavy-tailed noise can induce implicit compressibility—a property potentially lost in adaptive schemes involving clipping or normalization. Nevertheless, the theoretical understanding of vanilla \algname{SGD} in the heavy-tailed setting remains limited. The only existing convergence result, due to \cite{wang2021convergence}, relies on restrictive technical assumptions and fails to achieve optimal rates; see \Cref{sec:SC_SGD} for further discussion.

\textbf{Lower bounds and negative results.} Several works provide sample complexity lower bounds for first-order algorithms in different heavy-tailed regimes. For Lipschitz, convex functions on a bounded domain, the seminal works \citep{nemirovskij_yudin_1979_eff,raginsky2009information} provide a tight $\Omega\pare{\eps^{-p/(p-1)}}$ lower-bound. For $L$-smooth strongly-convex and non-convex functions, the $\Omega(\eps^{-\frac p {2(p-1)}})$ and $\Omega(\eps^{-\frac{3p - 2}{p-1}})$ lower-bounds for the class of first-order methods are established in \citep{zhang2020adaptive}. The works \cite[{\color{black}{Theorem 2.1}}]{sadiev2023high} and \cite{PriceAdaptivityStochastic2024Carmon} show in probability lower bounds, but their construction uses bounded noise and is limited to the bounded variance case. In comparison, our in probability lower bound construction (\Cref{thm:SGD_HP_lower_bound}) {\color{black}{is nearly tight for any $p\in (1, 2)$ and applies to methods whose output is a (fixed-coefficient) linear combination of the observed stochastic gradients}\allowbreak}. Under light-tailed noise, algorithm-specific lower-bounds are established for \algname{SGD} in \citep{drori2020complexity} and for stochastic approximation in \citep{khodadadian2025general}. The lower bounds for Hölder-smooth functions are available in \citep{guzman2015hoeldersmoothlowerbound,doikov2022lower,bai2025tight}. {\color{black}{We complement these works with \algname{SGD}-specific lower bounds under both $p$-Hölder smoothness with $p$-BM and standard smoothness with $p$-BCM. The latter applies to every deterministic positive non-increasing step-size schedule and exposes the polynomial confidence dependence of \algname{SGD}.}\allowbreak}

\section{Convex Setting}
Our main running assumption in this section is standard.
\begin{assumption}\label{ass:convex}
	The objective function $F(\cdot)$ is convex on $\cX \subseteq \R^d$ and there exists an optimizer $x^* \in \argmin_{y\in \cX} F(y) \subseteq \cX$.
\end{assumption}

\subsection{Upper Bound for SGD}
To analyze \algname{SGD}, we consider the Lyapunov (potential) function
$
\mathbb{E}\norm{x_t - x^*}^p,$ generalizing the classical analysis with $\mathbb{E}\norm{x_t - x^*}^2 .
$
We choose this new potential since the standard one with the power $2$ does not have to be bounded even after making the first step of \algname{SGD}, while the $p$-th power is bounded and is hence a suitable choice. The key technical challenge of the analysis is to build up an appropriate recursion for the $p$-th power of the norm. This is not straightforward since the direct argument of unrolling the square of the Euclidean norm by using inner product does not apply anymore. In order to overcome this challenge, our main observation is that the Euclidean norm raised to the power $p$, $p > 1$ (i.e., $\norm{x}^p$) is $(p-1)$-Hölder smooth \cite[Theorem 6.3]{Rodomanov2020}{\color{black}{, see \Cref{lem:hoelder_smoothness_of_eucl_norm}}}.
\begin{tcolorbox}[colback=white,colframe=mygreen!75!black]
\begin{theorem}\label{thm:convex_SGD}
	Let Assumptions~\ref{assum:pBM} and \ref{ass:convex} hold (\hyperref[assum:pBM]{(p-BM)} and convexity) with $p\in (1, 2]$, and \algname{SGD} is run with non-negative step-sizes $\cb{\eta_t}_{t\geq 1}$. %
    For any $t\geq 1$, define the weights $w_t \eqdef \eta_t \norm{x_t - x^*}^{p-2}$. Then for any $T\geq 1$
	$$
    \frac{\sum_{t=1}^T \Exp{w_t (F(x_t) - F(x^*)) }}{\sum_{t=1}^T \Exp{w_t} }
    \leq \frac{ \norm{x_1 - x^*}^2 + 4 G^2 \rb{\sum_{t=1}^T \eta_t^p}^{\nicefrac{2}{p}}}{\sum_{t=1}^T \eta_t }  .
    $$
    If, additionally, the set $\cX$ is bounded with diameter $D_{\cX},$  then for any $T\geq 1$
    $$
 \Exp{ F(\widetilde x_T) - F(x^*) } \leq  \frac{D_{\cX}^{2-p}\rb{ \norm{x_1 - x^*}^p + 4 G^p \sum_{t=1}^T \eta_t^p } }{\sum_{t=1}^T \eta_t} , \quad \widetilde x_T \eqdef \frac{\sum_{t=1}^{T} \eta_t \, x_t }{\sum_{t=1}^{T} \eta_t} .
$$
\end{theorem}
\end{tcolorbox}
\begin{proof}
We start with the Hölder smoothness of $\norm{\cdot}^p$ {\color{black}{from \Cref{lem:hoelder_smoothness_of_eucl_norm}}}, i.e.,
\begin{equation}
		\norm{v+w}^p
		\leq \norm{v}^p + p \frac{\<v , w\>}{\norm{v}^{2-p}} + 2^{2-p}\norm{w}^p \qquad \text{for any } v, w \in {\color{black}{\R^d}}, v \neq 0 .\label{eq:Holder_convex}
\end{equation}
By non-expansiveness of the projection, using the update rule of \algname{SGD} and the above inequality with $v \eqdef x_t - x^*$, $w \eqdef {\color{black}{-}}\eta_t \, \nabla f(x_t, \xi_t)$, and assuming $x_t \neq x^*$,\footnote{We implicitly assume throughout that $x_t \neq x^*$ for any $t \leq T$, otherwise the problem is solved.} we have
	\begin{align}\label{eq:hoelder_smoothness_norm}\begin{split}
		\norm{\xtp - x^*}^p
		& = \norm{\Pi_{\cX}(\xt - \eta_t \nabla f(x_t, \xi_t) ) - \Pi_{\cX}(x^*)}^p  \leq \norm{\xt - x^* - \eta_t \, \nabla f(x_t, \xi_t) }^p \\
		& \overset{\eqref{eq:Holder_convex}}{\leq} \norm{\xt - x^*}^p - \eta_t p \frac{\<\nabla f(x_t, \xi_t), \xt - x^* \>}{\norm{\xt - x^*}^{2-p}} + 2^{2-p} \eta_t^p \norm{\nabla f(x_t, \xi_t) }^p .
	\end{split}\end{align}
	Next we take conditional expectation and use convexity \Cref{ass:convex} along with \hyperref[assum:pBM]{(p-BM)} to derive
	\begin{equation}
		\CondExp{\norm{\xtp - x^*}^p}{\xt}
		\leq \norm{\xt - x^*}^p - \eta_t p \frac{F(x_t) - F(x^*) }{\norm{\xt - x^*}^{2-p}} + 2^{2-p} \eta_t^p G^p . \label{eq:cond-recursion_convex}
	\end{equation}
Define $r_t \coloneqq \|x_t - x^*\|$, $\Delta_t \coloneqq F(x_t) - F(x^*)$, $w_t \coloneqq \eta_t \, r_t^{p-2}$. Taking the total expectation of \eqref{eq:cond-recursion_convex}, we obtain:
\begin{equation}\label{eq:convex_expect_rec}
\mathbb{E}[r_{t+1}^p]
\leq \mathbb{E}[r_t^p] - p \, \mathbb{E}\left[  w_t \Delta_t \right] + 2^{2 - p} \eta_t^p G^p.
\end{equation}
Ignoring the negative term on the RHS and unrolling, we establish for any $t\geq 1$
\begin{equation}\label{eq:convex_bound_distance}
    \Exp{r_t^p} \leq r_1^p + 2^{2-p} G^p \sum_{\tau=1}^{t-1} \eta_{\tau}^p \leq r_1^p + 2^{2-p} G^p \sum_{\tau=1}^{t} \eta_{\tau}^p \eqdefright C_t.
\end{equation}
Coming back to \eqref{eq:convex_expect_rec}, we derive the bound
\begin{eqnarray*}
\tfrac{\sum_{t=1}^T \Exp{w_t \Delta_t}}{\sum_{t=1}^T \Exp{w_t} } &\leq& \tfrac{ C_T }{\sum_{t=1}^T \mathbb{E}[w_t] }
= \tfrac{ C_T }{\sum_{t=1}^T \eta_t \mathbb{E}[ (r_t^p)^\frac{p-2}{p} ] }
\leq \tfrac{ C_T }{\sum_{t=1}^T \eta_t \mathbb{E}[ r_t^p ]^\frac{p-2}{p} } \leq \tfrac{ C_T }{\sum_{t=1}^T \eta_t C_T^{\frac{p-2}{p}} } = \tfrac{ C_T^{\nicefrac{2}{p}} }{\sum_{t=1}^T \eta_t },
\end{eqnarray*}
where the second inequality follows by convexity of $x\mapsto x^{\frac{p-2}{p}}$ along with Jensen's inequality, and the last inequality follows from \eqref{eq:convex_bound_distance} and the fact that $\frac{p-2}{p} \leq 0$. The final bound follows by recalling the definition of $C_T$ given in \eqref{eq:convex_bound_distance} and simplifying the rate using the fact that $p\in (1, 2].$

If, additionally, the diameter is bounded, then we can upper bound the denominator in \eqref{eq:cond-recursion_convex} and repeat similar steps to derive the second inequality in the theorem statement.
\end{proof}

\paragraph{Discussion} Notice that convergence is shown for an average suboptimality gap weighted with the dynamic random weights $w_t = \eta_t \norm{x_t - x^*}^{p-2}$. Unfortunately, due to the correlation between the suboptimality and the weights, it is challenging to characterize the convergence for a specific point $x_t$ in the sequence $\cb{x_t}_{t\geq 1}$ even if we knew the weights $\cb{w_t}_{t\geq 1}$. However, when the diameter of $\cX$ is bounded we can establish convergence for implementable quantities: the running average iterate $\widetilde x_T = \sum_{t=1}^{T} \eta_t x_t / \sum_{t=1}^{T} \eta_t $ or the simple average $\widetilde x_T = \frac{1}{T} \sum_{t=1}^T x_t$.

Now we discuss several choices for the step-size sequence. First, if we know all problem parameters, we can set the step-sizes diminishing as $\eta_t = \frac{\norm{x_1 - x^*}}{G \, t^{\nicefrac{1}{p}}}$ for all $t \geq 1$, or constantly $\eta_t \equiv \frac{\norm{x_1 - x^*}}{G \, T^{\nicefrac{1}{p}}}$ to obtain
    $
    \Exp{ F(\widetilde x_T) - F(x^*) } = \widetilde{\cO}\rb{\frac{G D_{\cX}^{2-p} \|x_1 - x^*\|^{p-1} }{T^{\frac{p-1}{p}}}}$
for all $T\geq 1$. The logarithmic term hidden in $\widetilde{\cO}$ disappears for constant step-size. This implies the  $\cO(\varepsilon^{-\frac{p}{p-1}})$ sample complexity, which is known to be optimal under Assumptions~\ref{assum:pBM} and \ref{ass:convex} for all $p\in (1, 2]$, see e.g., \citep[Chapter V, Section 3.1]{nemirovskij_yudin_1979_eff} or \citep{raginsky2009information}. Second, we explore a universal step-size strategy which does not require knowledge of any parameters, $\eta_t = 1/\sqrt{t}$ for all $t\geq 1$. In this case we have for all $T\geq 1$ the bound $
    \Exp{ F(\widetilde x_T) - F(x^*) }  \leq \widetilde{\cO}\rb{ \frac{D_{\cX}^{2-p} \|x_1 - x^*\|^2 }{\sqrt{T}} +  \frac{  D_{\cX}^{2-p} G^p }{T^{\frac{p-1}{2}}} } . $ This result shows that even untuned projected \algname{SGD} converges under heavy-tailed noise. While for $p=2$ this convergence rate is minimax optimal in terms of the dependence on $T$, it is suboptimal for all other moments $p\in (1, 2).$ In \Cref{sec:nonconvex.lower_bound} and \Cref{sec:experiments}, we explore theoretically and experimentally the sensitivity of \algname{SGD} to step-size misspecification and find that this convergence rates are tight for both above mentioned step-size strategies.

\paragraph{Comparison to prior work} Several existing adaptive algorithms in the literature achieve similar convergence rates as above. For example, a scheme based on the mirror descent framework was proposed in the seminal work \citep[Chapter V, Section 2.1]{nemirovskij_yudin_1979_eff} and later revisited in \citep{vural2022mirror,liu2023revisiting}. The update rule of this scheme, when simplified to the unconstrained Euclidean setup, can be written as
$$
\text{\algname{p-SMD}:} \qquad x_{t+1} = \frac{ x_t \norm{x_t}^{\frac{2-p}{p-1}} - \eta_t \nabla f(x_t, \xi_t)   }{ \norm{  x_t \norm{x_t}^{\frac{2-p}{p-1}} - \eta_t \nabla f(x_t, \xi_t)   }^{2-p} } , \qquad \eta_t = \frac{\eta_1}{t^{\nicefrac{1}{p}}}  .
$$
As we can see even after some conceptual simplification, this method is more complicated than vanilla \algname{SGD}. First, it can be classified as an adaptive/normalized method due to the normalization of the updated point. Second, it requires computing the norm of $x_t$ and another auxiliary vector. Third, it not only requires tuning the step-size $\eta_t$, but also requires setting the correct power for computing $\norm{x_t}^{\frac{2-p}{p-1}},$ which is critical for convergence. This leads to unnecessary complications for implementations and computational instabilities. In fact, the implementation of the projected variant of this method requires an additional non-Euclidean projection, which can add extra computational burden.

Another modification of \algname{SGD} that has become very popular in stochastic optimization and machine learning literature recently is called \algname{SGD} with gradient clipping or \algname{Clip-SGD} \citep{zhang2020adaptive,sadiev2023high,liu2023stochastic,zhang2022parameter}. For a predefined sequence of clipping thresholds $\cb{\lambda_t}_{t\geq 1}$, the update rule of this algorithm is
$$
		\text{\algname{Clip-SGD}: } \qquad x_{t+1} = \Pi_{\mathcal{X}} \left( x_t - \eta_t g_t \right), \quad g_t = \operatorname{clip}\left( \nabla f(x_t, \xi_t), \lambda_t \right)
$$
	with $\operatorname{clip}(v, \lambda) := v \cdot \min\left\{1, \lambda / \|v\|_2 \right\}$. The typical choice of the sequence $\cb{\lambda_t}_{t\geq 1}$ in the convex setting is increasing and has the order $\lambda_t = \lambda_1 t^{\nicefrac{1}{p}}$, see e.g., \cite{nguyen2023improved,sadiev2023high,liu2023stochastic}.\footnote{In fact the choice of exact sequence is more complicated and depends on other problem parameters, but we report the asymptotic behavior for intuition.} While this algorithm is simpler than the previous scheme, it still requires additional hyper-parameter tuning for sequence $\cb{\lambda_t}_{t\geq 1}$.

    To summarize, the main advantage of our analysis in \Cref{thm:convex_SGD} is achieving similar convergence rates to above mentioned methods with a simpler algorithm -- vanilla \algname{SGD}.

\subsection{Lower Bound in High Probability}\label{subsec:convex_high_probability_lower_bound}
The results in the previous section imply, in particular, that for any $p\in (1, 2]$, \algname{SGD} with constant step-size $\eta_t = \frac{ \|x_1 - x^*\| }{G \, T^{\nicefrac{1}{p}}}$ converges at the rate
    $
    \Exp{ F\rb{\widetilde x_T} - F(x^*) } \leq \frac{5\, G D_{\cX}^{2-p} \|x_1 - x^*\|^{p-1}  }{T^{\frac{p-1}{p}}} .
    $
While this rate is known to be optimal \cite{nemirovskij_yudin_1979_eff,raginsky2009information} in-expectation, it does not give us any insights about the behavior of an individual run of the algorithm. To measure the concentration of \algname{SGD} around this expected convergence rate, we can use the Markov's inequality, which implies with probability at least $1-\delta$:
\begin{equation}\label{eq:HP_upper_via_Markov}
 F\rb{\widetilde x_T} - F(x^*)  \leq \frac{5\,G D_{\cX}^{2-p} \|x_1 - x^*\|^{p-1}   }{T^{\frac{p-1}{p}}} \frac{1}{\delta}  , \qquad \text{where } \quad \widetilde{x}_T = \frac{1}{T}\sum_{t=1}^T  x_t .
\end{equation}
As we can see, this result has a poor dependence on $1/\delta$. {\color{black}{In this section, we consider methods whose output belongs to $x_1$ plus the linear span of the observed stochastic gradients, i.e., for any $T \geq 2$}\allowbreak}
\begin{equation}\label{eq:cone_method}
    {\color{black}{x_{T}^{out} \in x_1 + \operatorname{span}\cb{\nabla f(x_1,\xi_1),\ldots,\nabla f(x_{T-1},\xi_{T-1})}}}.
\end{equation}
\begin{sloppypar}
{\color{black}{More precisely, for every horizon $T$, the coefficients in the representation $x_T^{out}=x_1-\sum_{t=1}^{T-1}\gamma_t\nabla f(x_t,\xi_t)$ may be arbitrary real numbers, but the vector $(\gamma_t)_{t=1}^{T-1}$ is deterministic and fixed before the hard instance is selected and the algorithm is run. We will show that for such fixed-coefficient outputs, a similar polynomial dependence as in \eqref{eq:HP_upper_via_Markov} is inevitable.}\allowbreak}\footnote{Here we say ``similar'' because formally our upper bound is established under bounded diameter assumption, while the lower bound construction has an unbounded domain.} {\color{black}{This implies, in particular, such a dependence for \algname{SGD} with arbitrary predefined step-sizes and fixed linear-combination output rules.}\allowbreak}
\end{sloppypar}

\begin{tcolorbox}[colback=white,colframe=mygreen!75!black]
\begin{theorem}\label{thm:SGD_HP_lower_bound}
    Let {\color{black}{$p \in (1, 2)$}\allowbreak}, $T\geq 2$, and $\delta \in (0, \nicefrac{1}{8}]$. {\color{black}{For every $\alpha\in(p,2)$ and every algorithm satisfying the fixed-coefficient linear-span condition \eqref{eq:cone_method}, there exist}\allowbreak} a convex problem \eqref{eq:problem}{\color{black}{, a starting point, and an unbiased stochastic gradient oracle satisfying}} \Cref{assum:pBM} {\color{black}{such that, with probability at least $\delta$,}\allowbreak}
    $$
F(x_{T}^{out}) - F(x^*) \geq %
\frac{\norm{x_1 - x^*}}{2 (T-1)^{\frac{\alpha-1}{\alpha}}}  \rb{\frac{1}{4 \, \alpha \, \delta}}^{\nicefrac{1}{\alpha}}  = \Omega\rb{ \frac{G \norm{x_1 - x^*} }{ T^{\frac{\alpha - 1}{\alpha}}} \rb{\frac{1}{ \delta}}^{\nicefrac{1}{\alpha}} },
$$
where the last equality holds for
$
{\color{black}{T \geq 1 + \left[\frac{1}{2\sigma_{p,\alpha}}
\rb{\frac{1}{4\alpha\delta}}^{\nicefrac{1}{\alpha}}\right]^{\frac{\alpha}{\alpha-1}}}}
$
{\color{black}{with $\sigma_{p,\alpha}\eqdef(\Exp{|\xi|^p})^{1/p}$ for the stable random variable $\xi$ specified in the proof.}}
\end{theorem}
\end{tcolorbox}
\begin{proof}
We define a one-dimensional problem on $\cX = \R$ and the stochastic gradient oracle:
\[
\begin{array}{rl@{\hspace{0.7cm}}rl}
F(x) &=
\begin{cases}
- a \, x & \text{if } x \leq 0 \\
\frac{L}{2} x^2 - a\,x & \text{if } 0\leq x \leq \frac{a}{L} \\
-\frac{a^2}{2\,L}  & \text{if } x \geq \frac{a}{L} ,
\end{cases}
&
\nabla f(x,\xi) &=
\begin{cases}
- a + \xi & \text{if } x \leq 0 \\
L \, x - a + \xi & \text{if } 0\leq x \leq \frac{a}{L} \\
 \xi  & \text{if } x \geq \frac{a}{L} ,
\end{cases}
\end{array}
\]
where $a, L > 0$ will be specified later. {\color{black}{For $\alpha\in(p,2)$, let $\xi\in\R$ be the symmetric $\alpha$-stable random variable in the normalization of \cite[Theorem~7]{bednorz2018tails}, with characteristic function}\allowbreak}
\begin{equation*}
{\color{black}{\Exp{\exp\rb{i \,s \,\xi}} = \exp\rb{-\kappa_\alpha |s|^{\alpha}},
\qquad \kappa_\alpha \eqdef \frac{\pi}{\Gamma(1+\alpha)\sin(\pi\alpha/2)}}}.
\end{equation*}
{\color{black}{This distribution is zero-mean and has a finite $p$-th moment because $p<\alpha$. Set $\sigma_{p,\alpha}\eqdef(\Exp{|\xi|^p})^{1/p}$. Then $p$-BCM holds with $\sigma_{p,\alpha}$, and \hyperref[assum:pBM]{(p-BM)} holds with
$G \eqdef \left[2^{p-1}(a^p + \sigma_{p,\alpha}^p)\right]^{\nicefrac{1}{p}}$.
Let $N\eqdef T-1$, and at the oracle calls let $\xi_1,\ldots,\xi_N$ be independent copies of $\xi$. Define
$A_\alpha\eqdef\rb{\sum_{t=1}^N|\gamma_t|^\alpha}^{1/\alpha}$ and
$B\eqdef\sum_{t=1}^N|\gamma_t|$. We first consider $B>0$. Writing $\nabla f(x,\xi)=h(x)+\xi$, the construction gives $-a\leq h(x)\leq0$. Therefore, \eqref{eq:cone_method} implies}\allowbreak}
\begin{equation}\label{eq:x_out_unrolled}
\begin{aligned}
{\color{black}{x_T^{out}}}
&{\color{black}{=x_1-\sum_{t=1}^N\gamma_t h(x_t)-\sum_{t=1}^N\gamma_t\xi_t \leq x_1+aB-\sum_{t=1}^N\gamma_t\xi_t}}.
\end{aligned}
\end{equation}
{\color{black}{where we used $-\gamma_t h(x_t)\leq a|\gamma_t|$. Set}\allowbreak}
\begin{equation*}
{\color{black}{a \eqdef \frac{A_\alpha}{2B}\rb{\frac{1}{4\alpha\delta}}^{\nicefrac{1}{\alpha}},
\qquad L \eqdef \frac{1}{B}}}.
\end{equation*}
{\color{black}{By independence and the characteristic function above,}\allowbreak}
\begin{equation*}
\begin{aligned}
{\color{black}{\Exp{\exp\rb{is\sum_{t=1}^N\gamma_t\xi_t}}}}
&{\color{black}{=\exp\rb{-\kappa_\alpha|s|^\alpha\sum_{t=1}^N|\gamma_t|^\alpha}}}
{\color{black}{=\Exp{\exp\rb{isA_\alpha\xi}}}},\\
&{\color{black}{\text{and hence}\qquad
\sum_{t=1}^N\gamma_t\xi_t\overset{d}{=}A_\alpha\xi}}.
\end{aligned}
\end{equation*}
{\color{black}{Theorem~7 in \citep{bednorz2018tails} gives
$\Pr(\xi\geq y)\geq[2(2+\alpha y^\alpha)]^{-1}$ for $y>0$. Thus, since $\delta\leq1/8$,}\allowbreak}
\begin{equation*}
{\color{black}{\Pr\rb{\sum_{t=1}^N\gamma_t\xi_t
\geq A_\alpha\rb{\frac{1}{4\alpha\delta}}^{\nicefrac{1}{\alpha}}}}}
{\color{black}{\geq \frac{1}{2\rb{2+\frac{1}{4\delta}}}
=\frac{2\delta}{1+8\delta}\geq\delta}}.
\end{equation*}
{\color{black}{On this event, \eqref{eq:x_out_unrolled} and the definition of $a$ give
$x_T^{out}\leq x_1-aB<0$. Select an arbitrary $x_1<0$. The closest optimum is $x^*=a/L=aB$, and therefore}\allowbreak}
\begin{align*}
{\color{black}{F(x_T^{out})-F(x^*)}}
&{\color{black}{\geq -ax_1+\frac{3}{2}a^2B}}\\
&{\color{black}{=a\norm{x_1-x^*}+\frac{1}{2}a^2B
\geq a\norm{x_1-x^*}}}.
\end{align*}
{\color{black}{Moreover, Hölder's inequality yields}\allowbreak}
\begin{equation*}
\begin{aligned}
&{\color{black}{B\leq N^{\frac{\alpha-1}{\alpha}}A_\alpha}},\\
&{\color{black}{a\geq \frac{1}{2N^{\frac{\alpha-1}{\alpha}}}
\rb{\frac{1}{4\alpha\delta}}^{\nicefrac{1}{\alpha}}
\eqdef c_{N,\alpha,\delta}}}.
\end{aligned}
\end{equation*}
{\color{black}{This proves the explicit lower bound in the statement. Furthermore,
$a/G\geq\frac12\min\{1,a/\sigma_{p,\alpha}\}$. If
$c_{N,\alpha,\delta}\leq\sigma_{p,\alpha}$, which is equivalent to the stated lower bound on $T$, then}}
\begin{equation*}
{\color{black}{F(x_T^{out})-F(x^*)
\geq \frac{G\norm{x_1-x^*}}{2\sigma_{p,\alpha}}c_{N,\alpha,\delta}
=\Omega\rb{\frac{G\norm{x_1-x^*}}{T^{\frac{\alpha-1}{\alpha}}}
\rb{\frac1\delta}^{\nicefrac1\alpha}}}}.
\end{equation*}
{\color{black}{It remains to handle $B=0$. In this case $x_T^{out}=x_1$. Use the same construction with
$a=2c_{N,\alpha,\delta}$, $L=1$, and $x_1<0$. Then $x^*=a$ and, deterministically,}}
\begin{equation*}
{\color{black}{F(x_T^{out})-F(x^*)=-ax_1+\frac{a^2}{2}
\geq c_{N,\alpha,\delta}(a-x_1)
=c_{N,\alpha,\delta}\norm{x_1-x^*}}}.
\end{equation*}
{\color{black}{Under the same condition on $T$, $a/(2G)\geq c_{N,\alpha,\delta}/(4\sigma_{p,\alpha})$; hence the $\Omega$-form follows as above.}}
\end{proof}

{\color{black}{First, the fixed-coefficient condition includes}\allowbreak} \algname{SGD} with {\color{black}{last-iterate}\allowbreak} output, for $\gamma_t=\eta_t${\color{black}{, and}\allowbreak} \algname{SGD} with simple-average output \eqref{eq:HP_upper_via_Markov}, for {\color{black}{$\gamma_t=\eta_t(T-t)/T$}\allowbreak}. {\color{black}{With standard initialization and predefined linear scalar recurrences, the condition also includes}\allowbreak} heavy-ball momentum \citep{polyak1964some,Gadat_SHB_18} {\color{black}{and}\allowbreak} Nesterov's acceleration \citep{nesterov1,lan2012optimal}{\color{black}{. It does not include methods whose effective coefficients depend on the observed gradients, such as clipping or normalization.}\allowbreak}{\color{black}{}\allowbreak}.
In particular, this lower bound shows that the convergence rate of \algname{SGD} will necessarily be multiplied by a polynomial in the inverse failure probability $1/\delta,$ unlike recent results for high probability convergence of \algname{Clip-SGD} \citep{sadiev2023high,liu2023stochastic,zhang2022parameter} and \algname{Normalized-SGD} \citep{hubler2024gradient}. Moreover, {\color{black}{the lower bound holds for every $\alpha\in(p,2)$}\allowbreak}, which means that it is nearly tight when compared to the best upper bound achieved by \Cref{thm:convex_SGD} in terms of dependence on $T$ and $p$. This lower bound is also remarkably tight in $\delta$ dependence and only leaves a small gap compared to our upper bound \eqref{eq:HP_upper_via_Markov} of order $\delta^{1- \frac{1}{\alpha}}$ in failure probability, which {\color{black}{approaches one as $p$ and $\alpha$ approach one}\allowbreak}. It is worth to mention that previously Sadiev et al. \citep{sadiev2023high} established a high probability lower bound for \algname{SGD} in strongly convex setting using bounded noise, $\Omega\rb{\nicefrac{1}{\sqrt{\varepsilon \delta} }}$, which is not tight for their setting. While their dependence on $1/\delta$ is also polynomial, {\color{black}{our construction is different and applies to every method satisfying the fixed-coefficient condition \eqref{eq:cone_method}}\allowbreak}.

{\color{black}{The lower bound above is established on the unbounded domain}\allowbreak} {\color{black}{$\cX=\R$}\allowbreak} {\color{black}{and therefore does not by itself yield a matching lower bound under the bounded-diameter assumption of}\allowbreak} {\color{black}{}\allowbreak} \Cref{thm:convex_SGD}{\color{black}{.}\allowbreak} We believe it is possible to modify our construction to allow a bounded set, e.g., $\cX = [-D_{\cX}, 0]$, and replace the initial distance to the optimum, $\norm{x_1 - x^*}$, with a sufficiently large diameter, $D_{\cX}.$ This would allow us to formally match the upper bound of \algname{SGD} in \eqref{eq:HP_upper_via_Markov}. However, the two main obstacles are (i) {\color{black}{formulating an appropriate fixed-coefficient condition}\allowbreak} for the constrained setting and (ii) carefully selecting the diameter $D_{\cX}$ to avoid frequently hitting the projection on the left. We believe the second obstacle is manageable when considering the last iterate of projected \algname{SGD} instead of a general class of algorithms and selecting sufficiently large $D_{\cX}$. However, the obstacle (i) seems more challenging.

\section{Strongly Convex Setting}\label{sec:SC_SGD}
In this section, we revisit the convergence of (projected) \algname{SGD} for strongly convex functions. We first recall the definition.
\begin{assumption}\label{ass:strongly_convex}
    The objective function $F(\cdot)$ is $\mu$-strongly convex on $\cX \subseteq \R^d$, i.e., $F(\cdot) - \frac{\mu}{2}\sqnorm{\cdot}$ is convex on $\cX$ with $\mu > 0.$
\end{assumption}

Now we are ready to present the main result for strong convex \algname{SGD} followed by its proof.
\begin{tcolorbox}[colback=white,colframe=mygreen!75!black]
\begin{theorem}\label{thm:SCprojectedSGD_improved}
	Let Assumptions~\ref{assum:pBM} and \ref{ass:strongly_convex} (\hyperref[assum:pBM]{(p-BM)} and $\mu$-strong convexity) hold with $p\in (1, 2{\color{black}{]}}.$ Then the iterates generated by \algname{SGD} with step-size $\eta_t = \frac 2 {\mu \, t}$ satisfy for any $T \geq 1$
    $$
     \Exp{ (F(\bar x_T) - F(x^*))^{\nicefrac{p}{2}}  + \rb{\frac{\mu}{2}}^{\nicefrac{p}{2}}  \norm{x_{T+1} - x^*}^p } \leq
     \begin{cases}
\frac{8 G^p }{ (2-p) \, 2^{\nicefrac{p}{2}} \mu^{\nicefrac{p}{2}} \, T^{p-1}}, & {\color{black}{p \in (1, 2) , }}\\
{\color{black}{\frac{2 G^2 \, (1 + \ln T)}{\mu \, T},}} & {\color{black}{p = 2 ,}}
\end{cases}
    $$
    where $\bar x_T$ is sampled uniformly from the iterates $\cb{x_1, \ldots, x_T}.$
\end{theorem}
\end{tcolorbox}

\begin{proof}
We start similarly to our analysis in convex case using Hölder smoothness of $\norm{x}^{p}$ {\color{black}{(see Equation \eqref{eq:hoelder_smoothness_norm})}} to get
	\begin{equation*}{\color{black}{
		\norm{\xtp - x^*}^p
		{\leq} \norm{\xt - x^*}^p - \eta_t p \frac{\<\nabla f(x_t, \xi_t), \xt - x^* \>}{\norm{\xt - x^*}^{2-p}} + 2^{2-p} \eta_t^p \norm{\nabla f(x_t, \xi_t) }^p.}}
	\end{equation*}
Next we take the conditional expectation and use strong convexity to derive
 		\begin{align}
		&\CondExp{\norm{\xtp - x^*}^p}{\xt}
		\leq \norm{\xt - x^*}^p - \eta_t p \frac{\frac{\mu}{2} \norm{\xt - x^*}^{2} + F(x_t) - F(x^*) }{\norm{\xt - x^*}^{2-p}} \notag \\
        & \qquad\qquad + 2^{2-p} \eta_t^p \CondExp{\norm{\nabla f(x_t, \xi_t) }^p}{\xt} \notag\\
        &\qquad = \rb{1 - \frac{p\, \eta_t \mu }{2}} \norm{\xt - x^*}^p - \eta_t p \frac{F(x_t) - F(x^*) }{\norm{\xt - x^*}^{2-p}} + 2^{2-p} \eta_t^p \CondExp{\norm{\nabla f(x_t, \xi_t) }^p}{\xt} \notag \\
        &\qquad \leq \rb{1 - \frac{\eta_t \mu }{2}} \norm{\xt - x^*}^p - \eta_t  \rb{\frac{\mu}{2}}^{\frac{2-p}{2}} \rb{ F(x_t) - F(x^*) }^{\nicefrac{p}{2}} \notag \\
        & \qquad \qquad + 2^{2-p} \eta_t^p \CondExp{\norm{\nabla f(x_t, \xi_t) }^p}{\xt} ,
        \label{eq:cond-recursion-SC}
	\end{align}
where in the last step we used the quadratic growth condition $F(x) - F(x^*) \geq \frac{\mu}{2} \sqnorm{x-x^*}$ to bound the term in the denominator and the fact that $p \in (1, 2]> 1.$ Define the distance to the optimum by $r_t \coloneqq \|x_t - x^*\|$ and the suboptimality $\Delta_t \coloneqq F(x_t) - F(x^*).$ Then taking the total expectation of \eqref{eq:cond-recursion-SC}, using \hyperref[assum:pBM]{(p-BM)}, and setting the step-sizes as $\eta_t = \nicefrac{2}{\mu\, t},$ we have
\begin{equation*}
\mathbb{E}[r_{t+1}^p]
\leq \frac{t-1}{t} \mathbb{E}[r_t^p] - \eta_t \rb{\frac{\mu}{2}}^{\frac{2-p}{2}} \Exp{\Delta_t^{\nicefrac{p}{2}}} + \frac{ 2^{2} G^p}{\mu^p\, t^p } .
\end{equation*}
Multiplying both sides by $t\geq 1$, summing up over $t = 1, \ldots, T$, and rearranging, we have
\begin{equation*}
  \rb{\frac{\mu}{2}}^{\frac{2-p}{2}}  \sum_{t=1}^T \mathbb{E}[t \, \eta_t \cdot \Delta_t^{\nicefrac{p}{2}} ] + T \, \Exp{ r_{T+1}^p } \leq  \sum_{t=1}^T \frac{ 4 G^p}{\mu^p\, t^{p-1} } \leq
\begin{cases}
\frac{8 G^p}{(2-p) \, \mu^p \, T^{p-2}}, &  {\color{black}{p \in (1, 2),}} \\
{\color{black}{\frac{4 G^2 (1 + \ln T) }{\mu^2},}} &  {\color{black}{p = 2.}}
\end{cases}
\end{equation*}
Pluggining in $\eta_t = \nicefrac{2}{(\mu\, t)}$ and dividing both sides by $ (\nicefrac{\mu}{2})^{-\nicefrac{p}{2}}\, T $ we obtain
\begin{equation}
 \frac{1}{T} \sum_{t=1}^T \mathbb{E}[\Delta_t^{\nicefrac{p}{2}} ] + \rb{\frac{\mu}{2}}^{\nicefrac{p}{2}} \Exp{ r_{T+1}^p } \leq \begin{cases}
\frac{8 G^p \, \rb{\frac{\mu}{2}}^{\nicefrac{p}{2}} }{(2-p) \, \mu^p \, T^{p-1}}, & {\color{black}{p \in (1, 2) , }}\\
{\color{black}{\frac{4 G^2 \, \rb{\frac{\mu}{2}} \, (1 + \ln T)}{\mu^2 \, T},}} & {\color{black}{p = 2 ,}}
\end{cases}
\end{equation}
{\color{black}{which concludes the proof after canceling $\mu$ and recalling the definition of $\bar x_T$.}}
\end{proof}

\paragraph{Discussion} The above theorem shows that instead of the classical convergence in expectation of the function sub-optimality and the distance to the optimum squared, these quantities should be raised to the power $p/2$ for any $p\in (1, 2).$ This modification of convergence criterion is meaningful in heavy-tailed setting and helps to circumvent the non-convergence example shown in \Cref{sec:intro}. Indeed, for that quadratic example we have $\Exp{(F(x_2)- F^*)^{\nicefrac{p}{2}}} = 2^{\nicefrac{p}{2}} \eta_1^p \Exp{\xi_1^p} < \infty$ due to \Cref{assum:pBM}. We remark here that while \Cref{thm:SCprojectedSGD_improved} requires the \textit{non-central} \hyperref[assum:pBM]{(p-BM)} assumption, it can be extended to the central $p$-BCM version when smoothness is available allowing for unbounded domains and the variance reduction effect; we omit this extension for brevity due to the page limit. The convergence rate established in \Cref{thm:SCprojectedSGD_improved} is optimal up to a numerical constant and the above mentioned nuances related to convergence criterion. In particular, we can translate our guarantee to in probability result using Markov's inequality, which gives for any fixed probability $\delta \in (0, 1)$ the sample complexity of order $\cO\rb{ \varepsilon^{-\frac{p}{2(p-1)}} }$ (both in function sub-optimality and the squared distance to the optimum). This matches with the lower bound for first-order methods in \cite{zhang2020adaptive}. It is critical to highlight that the above theorem prescribes to use the classical step-size $\eta_t = \nicefrac{2}{ \mu t}$, which is independent of tail index $p.$ This implies that knowledge of $p$ is not required for achieving the optimal sample complexity. Another interesting aspect of \Cref{thm:SCprojectedSGD_improved} is that the convergence in function value is shown for an iterate uniformly sampled from the trajectory $\cb{x_t}_{t\leq T}$ rather than the average \citep{stich2019unified} or last iterate \citep{zamani2023exact,fontaine2021convergence}. This is unusual for convex/strongly convex optimization where the average iterate convergence is more common \citep{nemirovskij1983problem,lan2020first}. This may imply that the averaging is not efficient without bounded variance and additional sampling from iterates is required to average out the noisy iterates. We provide further experimental study of different output selection strategies for \algname{SGD} in \Cref{sec:experiments}.

\paragraph{Comparison to prior work}
We first compare our guarantee to the literature on \algname{Clip-SGD} \cite{zhang2020adaptive,liu2023stochastic} (method is described in the previous section). In strongly convex setting, their recommendation for step-size is $\eta_t = \nicefrac{4}{\mu (t+1)}$ and for clipping threshold sequence is $\lambda_t  = \max\cb{ 2 \max_{x\in \cX} \norm{\nabla F(x)}, G \, t^{1/p} }$, see e.g., Theorem 9 in \cite{liu2023stochastic}.  They obtain the rate in expected function sub-optimality
\begin{equation}\label{eq:rate_ClipSGD}
\Exp{F(\widetilde x_T) -  F^* + \mu \sqnorm{x_T - x^*}} = \cO\rb{ \frac{G^2}{\mu (T+1)^{\frac{2(p-1)}{p}} }  }
\end{equation}
for some weighed average point $\widetilde x_T$. There are two main differences compared to our result. First, our convergence rate is established for vanilla \algname{SGD} without clipping. \algname{Clip-SGD} has two parameter sequences which are important to tune, and the clipping threshold depends on the tail index $p$ to achieve the optimal rate. In comparison, our \algname{SGD} only has the standard step-size $\eta_t = \nicefrac{2}{\mu t},$ which is easy to implement when the strong convexity modulus is known. The second difference is in the convergence criterion used. Our convergence criterion, $\Exp{ (F(\bar x_T) - F(x^*))^{\nicefrac{p}{2}}  + \rb{\frac{\mu}{2}}^{\nicefrac{p}{2}}  \norm{x_{T+1} - x^*}^p },$ can be weaker than the one for \algname{Clip-SGD} in \eqref{eq:rate_ClipSGD}. This means that our new analysis allows to get rid of additional clipping threshold hyper-parameter at the price of a slightly weaker convergence measure.

Now we compare our result to \cite{wang2021convergence}, which studies convergence of vanilla \algname{SGD} under a similar infinite variance assumption. Their step-size sequence choice is arbitrarily close to the harmonic decay, i.e., $\eta_t = \eta_1/t^{a}$ for any $a\in (0,1), $ and the convergence rate is presented as $\Exp{\norm{x_t - x^*}^p} \leq {C(d, L, \mu, \sigma)}/{t^{a(p-1)}},$ where the constant $C(d, L, \mu, \sigma)$ hides the dependence on dimension, variance parameter, smoothness and strong convexity parameters. While this result looks similar to our \Cref{thm:SCprojectedSGD_improved}, unfortunately, it has several important limitations. First, their analysis assumes that $F(\cdot)$ is twice differentiable function with a uniformly bounded spectral norm of the Hessian matrix $\nabla^2 F(x)$ ($L$-smoothness). Second, they make an additional non-standard assumption about the uniform $p$-positive definiteness of the Hessian matrix. To explain this concept, we let $p \geq 1$ and $\mathbf{Q}$ be a symmetric matrix. Define the signed power of a vector $\mathbf{v} \in \mathbb{R}^d$ as:
$
\mathbf{v}^{\langle q \rangle} := (\operatorname{sgn}(v^1)|v^1|^q, \dots, \operatorname{sgn}(v^d)|v^d|^q)^\top.
$
Let $S_p = \{\mathbf{v} \in \mathbb{R}^d : \|\mathbf{v}\|_p = 1\}$ be the unit sphere in $\norm{\cdot}_p$ norm. We say that $\mathbf{Q}$ is \emph{$p$-positive definite} if for all $\mathbf{v} \in S_p$,
$
\mathbf{v}^\top \mathbf{Q} \mathbf{v}^{\langle p - 1 \rangle} > 0.
$
In particular, in the limit case $p\rightarrow1$ the assumption reduces to the diagonal dominance of the Hessian. Even simple functions such as $F(x) = x^\top \begin{pmatrix} 0.02 & -1 \\ -1 & 50.02 \end{pmatrix} x$ violated this assumption for all $p \leq 1.8$. Another limitation of this work is that the convergence rates suffers from a polynomial dimension dependence, making it unscalable with $d$ even when the main problem constants (such as smoothness constant, strong convexity parameter and the variance) are dimension independent.

\section{Non-Convex Setting}\label{sec:NC_SGD}
In this section, we will establish convergence rates for non-convex \algname{SGD} under \hyperref[assum:pBM]{(p-BM)}, and will construct a hard instance function and gradient oracle to show tightness of our upper bound. Before we proceed with our main result, we need to introduce some notions and additional assumptions for the non-convex case. First, we assume the function is lower bounded, which is a standard assumption in this setting.
\begin{assumption}\label{ass:lower_bounded_constr}
	The function $F: \cX \rightarrow
	\R$ is lower bounded by $F^* > -\infty$ on $\cX$.
\end{assumption}
Our previous sections worked in non-smooth settings and we want to continue this exposition to be consistent with previous results. However, it is known that a general non-smooth non-convex optimization is intractable for \algname{SGD} {\color{black}{(even in the deterministic setting)}} and other randomized schemes are required even in the absence of heavy-tailed noise \citep{jordan2023deterministic}. To find a compromise between consistency and intractability, we will work with weakly Hölder convex and Hölder smooth functions \citep{drusvyatskiy2019efficiency,yashtini2016global,dvurechensky2017gradient}. %
\begin{assumption}[Hölder smoothness]\label{ass:smooth}
	Let $\nu \in [1,2]$ and $\ell_\nu, L_\nu > 0$. The objective function $F: \cX \rightarrow \R$ is $(\ell_\nu, L_\nu, \nu)$-\textit{Hölder smooth} with curvature exponent $\nu$ on $\cX $, i.e.,
	$$
	- \frac{\ell_\nu}{\nu} \norm{x-y}^{\nu} \leq F(x) - F(y) - \langle \nabla F(y), x - y \rangle \leq  \frac{L_\nu}{\nu} \norm{x-y}^{\nu} \qquad \text{for all } x, y \in \cX.
	$$
    In the case $\nu = 2,$ we say $F(\cdot)$ is $(\ell, L)$-\textit{smooth} or simply smooth.
\end{assumption}
The upper curvature constant $L_\nu$ for any $\nu \in (1,2]$ can potentially be infinity and, in principle, we can handle such case as well using a suitable convergence measure, see e.g., \citep{Davis_Drusvyatskiy_2018}. However, we will focus on the case $L_\nu < \infty,$ since in this case we can define a meaningful convergence criterion, which automatically recovers the gradient norm squared in the unconstrained setting. Note that the Hölder smoothness with $\nu < 2$ is weaker than standard smoothness (corresponding to $\nu = 2$) whenever the set $\cX$ is bounded. Specifically, if smoothness holds with $L = L_2$, then Hölder smoothness holds with $L_\nu = \nu L_2 D_{\cX}^{2-\nu} / 2.$ Thus, using \Cref{ass:smooth} with $\nu = p \in (1, 2]$ is not limiting in the constrained case if our main focus is the dependence on $\varepsilon$ and $p,$ but it may not guarantee the tightest possible bound in smoothness constant $L=L_2$ if the function is smooth.

Now we introduce some useful definitions needed for our analysis. For any $F : \cX \rightarrow \R $ and a real $\rho> 0$, the Moreau envelope and its associated proximal operator (at a point $x \in \cX$) are defined respectively by
$$
F_{1/\rho}^{\nu}(x) := \min_{y\in \cX} \sb{ F(y) +  \frac{\rho}{\nu} \norm{y - x}^{\nu} } , \qquad
\hat x^{\nu} := \argmin_{y\in \cX} \sb{ F(y) + \frac{\rho}{\nu} \norm{y - x}^{\nu}  } .
$$
In our analysis of \algname{SGD} we will use Hölder smoothness \Cref{ass:smooth} and Moreau envelope with $\nu = p.$ {\color{black}{Later, in the standard-smooth analysis of \algname{SGD}, we use these concepts with $\nu = 2$. In case $\nu = 2,$ we will omit the superscript and denote $F_{1/\rho}(\cdot) := F_{1/\rho}^{2}(\cdot),$ $\hat x := \hat x^{2} .$}\allowbreak}

\paragraph{Convergence criterion.}
We measure progress in the non-convex setting using a generalized stationarity measure based on \textit{Forward-Backward Envelope (FBE)} \cite{fatkhullin2024taming}, which is tailored to possibly constrained and non-smooth settings.
\begin{definition}[FBE and Generalized Stationarity Measure]\label{def:FBE_GSM}
Let $F : \cX \to \mathbb{R}$ be differentiable and satisfy \Cref{ass:smooth} with a curvature exponent $\nu > 1$. Let $\rho > 0$ be a regularization parameter. The Forward-Backward Envelope (FBE) of order $\nu$ at a point $x \in \cX$ is defined as
\begin{equation}\label{eq:FBE_order_nu}
\mathcal{D}_{\rho}^{\nu}(x) := - \frac{\nu \, \rho^{\frac{1}{\nu-1}} }{\nu - 1} \min_{y \in \cX} Q_\rho^{\nu}(x, y), \qquad
Q_\rho^{\nu}(x, y) := \langle \nabla F(x), y - x \rangle + \frac{\rho}{\nu} \norm{y - x}^{\nu}.
\end{equation}
We define the associated \emph{stationarity measure}
\begin{equation}\label{eq:stat_measure}
S_{\rho}^{\nu}(x) := \left( \mathcal{D}_{\rho}^{\nu}(x) \right)^{\frac{2(\nu - 1)}{\nu}}.
\end{equation}
\end{definition}
In the unconstrained Euclidean case (i.e., $\cX = \mathbb{R}^d$ or when projection is inactive), the FBE reduces to the gradient norm to the appropriate power:
$
\mathcal{D}_{\rho}^{\nu}(x) = \norm{\nabla F(x)}^{\frac{\nu}{\nu - 1}}.
$ Thus our stationarity measure $S_{\rho}^{\nu}(x)$ reduces to $S_{\rho}^{\nu}(x) = \|\nabla F(x)\|^2 $ for any $\nu > 1$ when $\cX = \mathbb{R}^d$. On the other hand when the constraint $\cX$ is active, but the objective is smooth with $\nu = 2,$ we can relate this stationarity measure with other well established notions, e.g., it is known that $S_{\rho}^{2}(x) = \mathcal{D}_{\rho}^{\nu}(x) \geq 2 \rho^2 \sqnorm{x - x^+}, $ where $x^+ \eqdef \argmin_{y \in \cX} \langle \nabla F(x), y \rangle + \rho \norm{y - x}^{2},$ see, e.g., \citep{fatkhullin2024taming}.

\subsection{Upper Bound for SGD}\label{subsec:upper_NC_SGD}
Now we are set to present our convergence result for non-convex \algname{SGD} under \hyperref[assum:pBM]{(p-BM)}.
\begin{tcolorbox}[colback=white,colframe=mygreen!75!black]
\begin{theorem}\label{thm:NC_SGD_upper}
	Let Assumptions~\ref{assum:pBM}, \ref{ass:lower_bounded_constr}, and \ref{ass:smooth} (\hyperref[assum:pBM]{(p-BM)}, lower bounded and $(\ell_p, L_p, p)$-Hölder smooth) hold, denote $\Delta_1 \eqdef F(x_1) - F^*.$ Then for any $T\geq 1$ \algname{SGD} with arbitrary step-size sequence $\cb{\eta_t}_{t\geq 1}$ satisfies
	$$
	\frac{ 1 }{\sum_{t=1}^T \eta_t} \sum_{t=1}^T  \eta_t \Exp{ S_{\rho+L_p}^p(x_t) } \leq 16 \rb{\frac{p}{p-1}}^{\gamma} \frac{ \Delta_1 + 2 \, \rho \, G^p \sum_{t=1}^{T} \eta_t^p }{ \sum_{t=1}^T  \eta_t  } ,
	$$
    where $\gamma \eqdef \frac{2(p-1)}{p},$ and $\rho \eqdef \frac{2(L_p + 2\,\ell_p)}{p-1}. $
\end{theorem}
\end{tcolorbox}
\begin{proof}
We start with definition of Moreau envelope and the optimality of $\hat{x}_{t+1}^p$:
\begin{align}
		F_{1/\rho}^p  \left(x_{t+1}\right)  &= F\left(\hat{x}_{t+1}^p\right) + \frac{\rho}{{\color{black}{p}}}\norm{\hat{x}_{t+1}^p - x_{t+1}}^p
		\leq F\left(\hat{x}_{t}^p\right) + \frac{\rho}{{\color{black}{p}}}\norm{ \hat{x}_{t}^p - x_{t+1} }^p  \nonumber\\
        &= F\left(\hat{x}_{t}^p\right) + \frac{\rho}{{\color{black}{p}}}\norm{ \Pi_{\cX}(\hat{x}_{t}^p) - \Pi_{\cX}(x_{t} - \eta_t \nabla f(x_t, \xi_t)) }^p  \nonumber\\
        &\leq F\left(\hat{x}_{t}^p\right) + \frac{\rho}{{\color{black}{p}}}\norm{ ( \hat{x}_{t}^p - x_t) + \eta_t \nabla f(x_t, \xi_t) }^p  \nonumber ,
        \end{align}
        where in the last inequality above we used non-expansiveness of projection. Next, similarly to convex case we use the Hölder smoothness of $\norm{\cdot}^p$:
        \begin{align}
        F_{1/\rho}^p  \left(x_{t+1}\right) &\leq F\left(\hat{x}_{t}^p\right) + \frac{\rho}{{\color{black}{p}}}\norm{ x_t - \hat{x}_{t}^p }^p + \frac{\rho \, p \, \eta_t}{{\color{black}{p}}} \frac{\< \hat{x}_t^p - x_t ,  \nabla f(x_t, \xi_t) \>}{\norm{\hat{x}_t^p - x_t}^{2-p}} + \frac{2^{2-p}\, \rho\, \eta_t^p}{{\color{black}{p}}} \norm{ \nabla f(x_t, \xi_t) }^p \nonumber \\
        &= F_{1/\rho}^p  \left(x_{t}\right) + \frac{\rho \, p \, \eta_t}{{\color{black}{p}}} \frac{\< \hat{x}_t^p - x_t ,  \nabla f(x_t, \xi_t) \>}{\norm{\hat{x}_t^p - x_t}^{2-p}} + \frac{2^{2-p}\, \rho\, \eta_t^p}{{\color{black}{p}}} \norm{ \nabla f(x_t, \xi_t) }^p \nonumber.
\end{align}
On the other hand, using \Cref{ass:smooth} and \Cref{le:descent_FBE}, we have for $\rho > \ell_p$
\begin{eqnarray}
 \langle \nabla F(x_t), \hat x_t^p - x_{t}\rangle &\leq&  F(\hat x_t^p) - F(x_t) + \frac{\ell_p}{p}\norm{\hat x_t^p - x_t }^p \notag \\
		&=& F_{1/\rho}^p(x_t) - F(x_{t}) + \frac{\ell_p - \rho }{p} \norm{\hat x_t^p - x_t }^p  \notag
		 \leq - C \, \cD_{\rho+L_p}^{p}(x_t), %
	\end{eqnarray}
where we denote $C \eqdef \fr{p-1}{p \, (\rho + L_p)^{\frac{1}{p-1}}}.$ Combining the above two inequalities, and defining $\psi_t \eqdef \nabla f(x_t, \xi_t) - \nabla F(x_t)$, $\Delta_{t, p} \eqdef F_{1/\rho}^p(x_t) - F^*$, we attain
 \begin{eqnarray}\label{eq:main_recursion_Nonconvex_SGD}
     \quad \quad  \Delta_{t+1, \,p} \leq \Delta_{t, \,p} - \frac{\rho \, p \, \eta_t \, C}{ {\color{black}{p}} } \frac{\cD_{\rho + L_p}^p(x_t)}{\norm{\hat{x}_t^p - x_t}^{2-p}}  + \frac{2^{2-p}\, \rho\, \eta_t^p}{{\color{black}{p}}} \norm{ \nabla f(x_t, \xi_t) }^p + \frac{\rho \, p \, \eta_t}{{\color{black}{p}}} \frac{\< \hat{x}_t^p - x_t , \psi_t \>}{\norm{\hat{x}_t^p - x_t}^{2-p}} .
\end{eqnarray}
It follows from \Cref{le:PM_FBE_connection} that for any $\rho > \frac{L_p + 2\,\ell_p}{p-1} > \ell_p$
$$
\norm{\hat x_t^p - x_t}^{p} \leq \frac{C}{\rho - \frac{L_p + 2\,\ell_p}{p-1}}  \cD_{\rho+L_p}^{p}(x_t) .
$$
Thus we can upper bound the denominator of the negative term in \eqref{eq:main_recursion_Nonconvex_SGD}
\begin{eqnarray*}
        \Delta_{t+1, \,p} &\leq& \Delta_{t, \,p} - \frac{\rho \, \eta_t}{{\color{black}{p}}} \rb{\rho - \frac{L_p + 2\,\ell}{p-1}}^{\frac{2-p} p}  \rb{C\,\cD_{\rho+L_p}^p(x_t)}^{\frac{2(p-1)}{p}}  + 2\, \rho\, \eta_t^p  \norm{ \nabla f(x_t, \xi_t) }^p \\
        &&\qquad + \frac{\rho \, p \, \eta_t}{{\color{black}{p}}} \frac{\< \hat{x}_t^p - x_t , \psi_t \>}{\norm{\hat{x}_t^p - x_t}^{2-p}} ,
\end{eqnarray*}
{\color{black}{where we also used the fact that $\nicefrac{2^{2-p}}{p} \leq 2 $ for $p\in [1,2].$}} Telescoping, taking the total expectation, using \Cref{assum:pBM} and rearrange the bound with $\rho = \frac{2(L_p + 2\,\ell_p)}{p-1}$ yields
\begin{equation*}
    \pare{\frac{\rho}{2}}^{\frac 2 p} C^{\frac{2(p-1)}p}\sum_{t=1}^T \eta_t S_{\rho + L_p}^p(x_t) \leq \Delta_{1,\, p} + 2 \rho G^p \sum_{t=1}^T \eta_t^p,
\end{equation*}
where we used $S_{\rho+L}^p(x_t) = (\mathcal{D}_{\rho}^{p}(x))^{\frac{2(p-1)}{p}}$. Noticing that $\Delta_{1, \,p} \leq \Delta_1 = F(x_t) - F^*,$ and $\pare{\frac \rho 2}^{\frac 2 p} C^{\frac{2(p-1)}{p}} \geq \pare{\frac{p-1}p}^{\frac{2(p-1)}p} \frac 1 {16}$ yields the claim.
\end{proof}

\paragraph{Discussion} If $\cX = \R^d,$ we recall that $S_{\rho}^p(x) = (\mathcal{D}_{\rho}^{p}(x))^{\frac{2(p-1)}{p}}  = \sqnorm{\nabla F(x)} $ for any $\rho > 0$ and any $p\in (1, 2].$ Therefore, the above theorem implies in the unconstrained case
$$
	 \mathbb E \, \sqnorm{\nabla F(\bar x_T)} \leq 16 \rb{\frac{p}{p-1}}^{\gamma} \frac{ \Delta_1 + 2 \, \rho \, G^p \sum_{t=1}^{T} \eta_t^p }{ \sum_{t=1}^T  \eta_t  } = \cO\rb{\frac{(\ell_p + L_p)^{\nicefrac{1}{p}} \Delta_1^{\frac{p-1}{p}} G }{T^{\frac{p-1}{p}}}},
$$
where the last equality holds setting $\eta_t = \sqrt[p]{\Delta_1}/ G \sqrt[p]{\rho \, T}$ and $\bar x_T$ is sampled from the iterates $\cb{x_t}_{t\leq T}$ either uniformly or with probabilities proportional to $\eta_t.$ In any of these cases, perhaps surprisingly \algname{SGD} converges in the standard measure for unconstrained non-convex optimization: expectation of the gradient squared. We remark that while our Hölder smoothness \Cref{ass:smooth} is weaker that standard smoothness (with $\nu = 2$) on any compact set $\cX$, there is still no contradiction with the folklore quadratic example discussed in \Cref{sec:intro}. The crux is that to include the quadratic example $F(x) = \frac{1}{2}x^2$ in our function class satisfying $(0, L_p, p)$-Hölder smoothness, we need to assume a bounded domain, which gives $L_p \leq p D_{\cX}^{2-p} / 2.$ This means the derivation in \eqref{eq:quadratic_example} is not valid anymore due to the projection on the bounded set $\cX,$ and the projected \algname{SGD} actually converges in terms of expectation of gradient norm squared.

Now we consider different step-size strategies. If we use the classical parameter agnostic step-sizes $\eta_t = 1/\sqrt{t}$ \citep{yang2023two_sides}, then after $T = \cO\rb{\varepsilon^{-\frac{4}{p-1}}}, $ \algname{SGD} finds a point $x \in \cX$ with $\mathbb E[\sqnorm{\nabla F(x)} ] \leq \varepsilon^2.$ If we use the $p$-dependent step-size, e.g.,  $\eta_t = \sqrt[p]{\Delta}/ G \sqrt[p]{\rho \, t}$, the sample complexity improves to $ \widetilde{\cO}\rb{\varepsilon^{-\frac{2p}{p-1}}}. $ We will show in the next section that both above mentioned complexities are tight for \algname{SGD} under our exact assumptions up to a numerical constant.

\paragraph{Comparison to prior work} To the best of our knowledge, there are no algorithms in the literature working precisely under the same assumptions as our \Cref{thm:NC_SGD_upper}. However, we can compare our result to other algorithms, which use the standard smoothness and the bounded central moment assumptions. We first compare to \algname{Clip-SGD} \citep{zhang2020adaptive,sadiev2023high}. Under smoothness and $p$-BCM, the typical recommended order of the step-size and clipping thresholds \algname{Clip-SGD} are $\eta_t = \eta_1 / t^{\frac{p}{3p-2}}$, $\lambda_t = \lambda_1 \, t^{\frac{1}{3p-2}}$. The above mentioned works derive the sample complexity of order $\cO\rb{\varepsilon^{-\frac{3p-2}{p-1}}}, $ which is smaller that our best possible result $ {\cO}\rb{\varepsilon^{-\frac{2p}{p-1}}} $ for \algname{SGD} with $\eta_t = \eta_1 / \sqrt[p]{T}$ step-size when $p < 2$, and matches when $p = 2.$ Several other adaptive methods achieve sample complexities similar to $\cO\rb{\varepsilon^{-\frac{3p-2}{p-1}}}.$ For example, \algname{Normalized SGD} with mini-batch or momentum \citep{hubler2024gradient,liu2024nonconvex}, and \algname{Normalized SGD} with gradient clipping and momentum \citep{cutkosky2021high}. However, all these algorithms are more complex than vanilla \algname{SGD}, rely on smoothness assumptions, are limited to unconstrained setting, and, similar to \algname{Clip-SGD}, require at least two hyper-parameters to achieve reduced sample complexity. For example, \algname{Normalized SGD} with momentum \citep{hubler2024gradient} requires step size $\eta_t = \eta_1 \sqrt{\alpha_t / t}$ and momentum $\alpha_t = \alpha_1 / t^{\frac{p}{3p - 2}}$.

In summary, we prove for the first time that non-convex \algname{SGD} converges under heavy-tailed noise. Our complexity bound is worse in dependence on $\varepsilon$ than those discussed above in the standard smooth setting for $p\in(1,2)$, but the assumptions and the algorithms are different. {\color{black}{This discrepancy can be caused by two possibilities:}\allowbreak}
\begin{itemize}
    \item[i.] Our analysis may be not tight.
    \item[ii.] \algname{SGD} is inherently slower than other adaptive methods such as \algname{Clip-SGD} even if $\nu = 2.$
\end{itemize}
{\color{black}{In the next section, we exclude possibility i by constructing an
algorithm-specific lower bound for \algname{SGD} with arbitrary polynomial
step-sizes, showing that the upper bound in \mbox{\Cref{thm:NC_SGD_upper}} is tight
under $p$-Hölder smoothness and $p$-BM. We address possibility ii in
\mbox{\Cref{subsec:mini_batch_SGD,subsec:Standard_smoothness_lower_bound}} under
standard smoothness and $p$-BCM. The matching upper and lower bounds show that
\algname{SGD} is indeed slower than adaptive methods when $p<2$.}\allowbreak}

\subsection{Lower-Bound for SGD}\label{sec:nonconvex.lower_bound}
When choosing polynomially decaying stepsizes $\eta_t = \eta \, t^{-r}$, $\eta > 0, r \in [0, 1)$, \Cref{thm:NC_SGD_upper} implies a sample complexity of
\begin{equation}\label{eq:nonconvex.lower_bound.sgd_polynomial_stepsize_sample_complexity_upper_bound}
    T = \tilde{\Oc}\pare{
        \pare{\frac{\Delta_1}{\eta \eps^2}}^{\frac{1}{1-r}}
        + \eta^{\frac 1 r} \pare{\frac{L_\nu G^p}{\eps^2}}^{\frac 1 {r(p-1)}}
    }
\end{equation}
to reach an $\eps$-stationary point, i.e., $\Exp{S_{\rho}^p} \leq \varepsilon$. As the previously established lower-bounds in the literature do not cover our set of assumptions, we derive the tightness of this result in all parameters in the following Theorem.
\begin{tcolorbox}[colback=white,colframe=mygreen!75!black]
\begin{theorem}\label{thm:nonconvex.lower_bound.SGD_lower_bound}
    Let $p, \nu \in (1, 2], \Delta_1, L_\nu \geq 0, 0 \leq \sigma \leq G, \eps > 0$ and $\eta > 0, r \in [0, 1)$. Assume $\eps < \nicefrac G 2$ and $\eps^{\frac{\nu} {\nu-1}} \leq \frac{\nu} {\nu-1} \frac{\Delta_1}{4}\pare{\frac{L_\nu}{2}}^{\frac 1 {\nu-1}}$.\footnotemark\ Then, for any dimension $d \in \Nb_{\geq 1}$ and $i \in \set{1,2}$, there exist
    convex, $(0, L_\nu, \nu)$-smooth and $G$-Lipschitz functions $F_1, F_2 \colon \R^d \to \R$ with $F_i(x_1) - \inf_{x \in \R^d} F_i(x) \leq \Delta_1$, and
    gradient oracles $\nabla f_i(x, \xi)$, that satisfy \hyperref[assum:pBM]{(p-BM)} and \hyperref[assum:pBCM]{(p-BCM)}
    such that \algname{SGD} with stepsizes $\eta_t = \eta \, t^{-r}$ almost surely requires at least
\begin{equation}\label{eq:nonconvex.lower_bound.SGD_lower_bound.lower_bound}
        T \geq \max\set{\pare{\frac{\pare{1-r}\Delta_1}{8 \eta \eps^2}}^{\frac 1 {1-r}},
            \eta^{\frac 1 r}
            \pare{\frac{\sigma^p}{2^pL_\nu}}^{\frac 1 {r\pare{p-1}}}
            \pare{\frac{L_\nu}{2\eps}}^{\frac{p - 1 + \nu - 1}{r\pare{p-1}\pare{\nu-1}}}
        }
    \end{equation}
    iterations to reach an $\eps$-stationary point.
\end{theorem}
\end{tcolorbox}
\footnotetext{{Note that both assumption are mild in the sense that, whenever one of them is violated, we have convergence at $x_1$ or within constantly many iterations.}}

In other words, the above theorem implies that for all $T$ that do not satisfy the above inequality \eqref{eq:nonconvex.lower_bound.SGD_lower_bound.lower_bound}, we have $\min_{t \in [T]} \norm{\nabla F_1(x_t)} > \eps$ or $\min_{t \in [T]} \norm{\nabla F_2(x_t)} > \eps$ almost surely. In particular, the $\eps$-dependence is given by
    \begin{equation*}
        T = \Omega \pare{\eps^{-\frac 2 {1-r}} + \eps^{-\frac{p - 1 + \nu - 1}{r\pare{p-1}\pare{\nu-1}}}}.
    \end{equation*}

The proof of \Cref{thm:nonconvex.lower_bound.SGD_lower_bound} consists of two lower-bound constructions, providing the first and second lower-bound term respectively. The following proposition provides the first term in \Cref{eq:nonconvex.lower_bound.SGD_lower_bound.lower_bound}, using a deterministic function construction punishing small stepsizes.

\begin{proposition}\label{prop:nonconvex.sgd_lb.deterministic_small_steps}
    Let $\nu \in (1,2], \eps > 0, \Delta_1, L_\nu \geq 0$ and assume the target accuracy $\varepsilon$ is sufficiently small, $\eps^{\frac{\nu} {\nu-1}} \leq \frac{\nu} {\nu-1} \frac{\Delta_1}{4}\pare{\frac{L_\nu}{2}}^{\frac 1 {\nu-1}}$. Then there exists a convex, $(2\eps)$-Lipschitz continuous, $(0, L_\nu, \nu)$-Hölder smooth function $F \colon \R \to \R$ with $F(x_1) - \inf_{x \in \R} F(x) \leq \Delta_1$ such that \algname{SGD} with stepsizes $\pare{\eta_t}_{t \in \Nb_{\geq 1}} \geq 0$ requires
    \begin{equation*}
        \sum_{t=1}^{T-1} \eta_t > \frac{\Delta_1}{4 \eps^2}
    \end{equation*}
    to hold to reach an $\eps$-stationary point, i.e.\ $x_T \in \R$ with $\norm{\nabla F(x_T)} \leq \eps$.
\end{proposition}
The proof of \Cref{prop:nonconvex.sgd_lb.deterministic_small_steps} builds on the idea of the hard instance construction from \citep{NSGDM_LzLo2023Huebler}. The key difference is that we adapt their construction {\color{black}{from \algname{Normalized SGD}}} to \algname{SGD} and {\color{black}{extend it from standard smoothness to}} Hölder-smoothness. Similar constructions under standard smoothness also appear in \citep[Section 3.2]{PerformanceFirstOrderMethods2014Drori}. %
The full proof of \Cref{prop:nonconvex.sgd_lb.deterministic_small_steps} can be found in \Cref{sec:app.missing_proofs}.

Next we focus on the second term of \eqref{eq:nonconvex.lower_bound.SGD_lower_bound.lower_bound}, by constructing a function and oracle that punish large stepsizes.

\begin{proposition}\label{prop:nonconvex.sgd_lb.stoch_large_steps}
    Let $\nu \in (1, 2], 0 \leq \sigma \leq G, \eta, L_\nu > 0$ and $x \in \R^d$. Furthermore let $d \in \Nb$ be arbitrary and
    \begin{equation*}
        F\colon \R^d \to \R, F(z) \coloneqq \begin{cases}
            \frac{L_\nu}{2^{2-\nu}\nu}\norm{z}^\nu, & \norm{z} \leq r_H\\
            \frac G 2 \norm{z} - C, & \norm{z} > r_H,
        \end{cases}
    \end{equation*}
    where  $r_H \coloneqq \pare{\frac{G}{2^{\nu-1}L_\nu}}^{\frac 1 {\nu - 1}}$ and $C = \frac{\nu - 1}{4\nu} \pare{\frac{G^\nu}{L_\nu}}^{\frac 1 {\nu - 1}}$. Then there exist a gradient oracle $\nabla f(z, \xi)$ that satisfies \hyperref[assum:pBM]{(p-BM)} and \hyperref[assum:pBCM]{(p-BCM)} such that $y \coloneqq x - \eta \nf{x, \xi}$ satisfies
    \begin{equation*}
        \norm{y} \geq \min\set{\norm{x}, \frac 1 2 \pare{\frac{\eta^{p-1}\sigma^p}{2^pL_\nu}}^{\frac 1 {p + \nu - 2}}}.
    \end{equation*}
\end{proposition}

\begin{figure}[t]
	\centering
	\definecolor{LBInk}{HTML}{17201F}
	\definecolor{LBSlate}{HTML}{4D5A58}
	\definecolor{LBTeal}{HTML}{2A9D8F}
	\definecolor{LBBlue}{RGB}{33,92,175}
	\definecolor{LBCoral}{HTML}{C85B45}
    \tikzset{>={Latex[width=1.5mm,length=2.5mm]}}
	\begin{tikzpicture}
		\pgfmathsetmacro{\L}{0.1}             %
		\pgfmathsetmacro{\G}{0.13}             %
		\pgfmathsetmacro{\nu}{1.7}               %
		\pgfmathsetmacro{\sig}{0.2}           %
		\pgfmathsetmacro{\p}{2}             %
		\pgfmathsetmacro{\stepsize}{1.4 / \L}           %
		\pgfmathsetmacro{\thresh}{%
			pow(pow(\stepsize,\p-1) * pow(\sig, \p)/(pow(2,\p) * \L), 1/(\p+\nu-2)) / 2%
		}
		\pgfmathsetmacro{\x}{(-\thresh)*1.4}             %
		\pgfmathsetmacro{\xrange}{abs(\x) * 1.05}   %
		\pgfmathsetmacro{\rH}{pow(\G / (pow(2,\nu - 1) * \L), 1/(\nu-1))}
		\pgfmathsetmacro{\C}{\G/2*\rH - \L/(pow(2,2-\nu)*\nu)*pow(\rH,\nu)}
		\pgfmathtruncatemacro{\drawLinear}{%
			ifthenelse(\rH<\xrange,1,0)%
		}
		\pgfmathsetmacro{\nFatx}{%
			ifthenelse(abs(\x)==0,0,%
				ifthenelse(abs(\x)<=\rH,%
					\L / pow(2, 2-\nu) * \x * pow(abs(\x),\nu-2),%
					\G / 2 *sign(\x)%
				)%
			)%
		}
		\pgfmathsetmacro{\yd}{\x - \stepsize * \nFatx}
		\pgfmathdeclarefunction{F}{1}{%
			\pgfmathparse{%
				ifthenelse(abs(#1)<=\rH,
				\L/(pow(2,2-\nu)*\nu)*pow(abs(#1),\nu),
				\G/2*abs(#1) - \C)}%
		}
		\pgfmathsetmacro{\Fx}{F(\x)}              %
		\pgfmathsetmacro{\Fyd}{F(\yd)}            %
		\pgfmathsetmacro{\Ftau}{F(\thresh)}       %
		\pgfmathsetmacro{\ymin}{-(0.25 * \Fx)}               %
		\begin{axis}[
			xmin=-\xrange, xmax=\xrange,                         %
			scale only axis,                                     %
			width=\axisdefaultwidth,                             %
			height=4cm,                                          %
			axis lines=middle,                                   %
			xlabel={$z$}, ylabel={$F(z)$},                       %
			axis line style={draw=LBSlate!75},
			tick style={draw=LBSlate!75},
			xlabel style={above,text=LBSlate},
			ylabel style={below right,text=LBSlate},              %
			xtick={-\thresh,\thresh,\x},                         %
			xticklabels={\textcolor{LBCoral}{$-\tau$},\textcolor{LBCoral}{$\tau$},\textcolor{LBBlue}{$x$}}, %
			x tick label style={anchor=south west,text=LBSlate},	 %
            extra x ticks={\yd},								 %
			extra x tick label={\textcolor{LBBlue}{$y_d$}},			 %
			extra x tick style={xticklabel style={below,text=LBSlate}}, %
			ytick=\empty,                                        %
			clip=false                                           %
			]

			\addplot[draw=LBInk, very thick, domain=-\xrange:\xrange, samples=400] {F(x)};

			\draw[solid, thick, ->, draw=LBBlue] (axis cs:\x,\Fx) -- (axis cs:\yd,\Fyd);
            \filldraw[LBBlue] (axis cs:\x,\Fx) circle (1.8pt); %

			\draw[dashed, very thick, ->, draw=LBCoral] (axis cs:\yd,\Fyd) -- (axis cs:-\thresh,\Ftau);
			\draw[dashed, very thick, ->, draw=LBCoral] (axis cs:\yd,\Fyd) -- (axis cs:\thresh,\Ftau);
            \filldraw[LBCoral] (axis cs:-\thresh,\Ftau) circle (1.8pt);     %
            \filldraw[LBCoral] (axis cs:\thresh,\Ftau) circle (1.8pt);      %

			\pgfplotsinvokeforeach{\x,-\thresh,\thresh}{%
				\pgfmathparse{F(#1)}\let\Fz\pgfmathresult
				\edef\temp{
					\noexpand\draw[LBSlate!55,dotted](axis cs:#1,\ymin) -- (axis cs:#1,\Fz);
				}
				\temp
			}

			\draw[LBSlate!75,|-|] (axis cs:\x,{\ymin/1.75}) -- (axis cs:-\thresh,{\ymin/1.75})
			node[pos=.5, fill=white, text=LBSlate, inner sep=1.5pt] {$\eta g_{-}$};
			\draw[LBSlate!75,|-|] (axis cs:\x,\ymin) -- (axis cs:\thresh,\ymin)
			node[pos=.5, fill=white, text=LBSlate, inner sep=1.5pt] {$\eta g_{+}$};

		\end{axis}

	\end{tikzpicture}
	\caption{Visualisation of the lower bound construction in \Cref{prop:nonconvex.sgd_lb.stoch_large_steps}, for the case where the deterministic update $y_d = x - \eta \nabla F(x)$ lands within the critical radius $\tau$. The dashed lines represent the two possible offsets introduced by the constructed gradient oracle, resulting in $y = x - \eta g_{\pm} u = x + \eta g_{\pm} = \pm \tau$.}
	\label{fig:temp}
\end{figure}
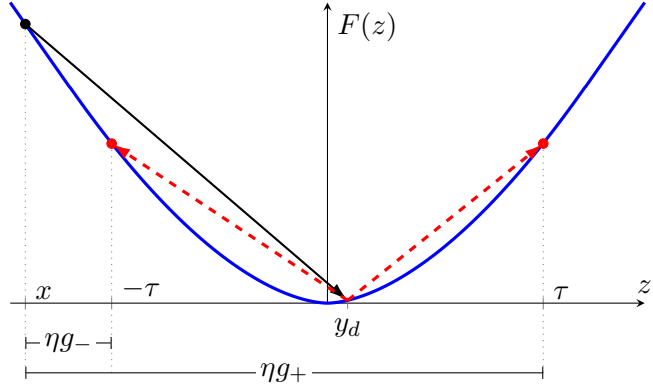

\begin{proof}
    We first note that $F$ is differentiable, and its gradient satisfies
    \begin{equation*}
        \norm{\nabla F(x)} = \left\{\begin{array}{lr}
            \frac {L_\nu} {2^{2 - \nu}}\norm{x}^{\nu - 1}, & \norm x \leq r_H\\
            \frac G 2, & \norm x > r_H
        \end{array}\right\}
        \leq \frac G 2.
    \end{equation*}
    The proof hinges on constructing a gradient oracle that prevents $y$ from entering the open ball $B_{\bar \tau}$ with radius $\bar \tau \coloneqq \min\set{\norm{x}, \tau}$, where $\tau \coloneqq \frac 1 2 \pare{\frac{\eta^{p-1}\sigma^p}{2^pL_\nu}}^{\frac 1 {p + \nu - 2}}$. We will differentiate two cases.

    \paragraph{Case I: $\norm{x - \eta\nabla F(x)} \geq \bar \tau$.} In this case we can use the deterministic gradient oracle $\nf{x, \xi} = \nabla F(x)$, as the deterministic update already lands outside $B_{\bar \tau}$. This oracle trivially satisfies $p$-BCM, and $p$-BM is satisfied due to $\norm{\nabla F(x)} \leq \frac G 2$.

    \paragraph{Case II: $\norm{x - \eta\nabla F(x)} < \bar \tau$.} This case means that the deterministic update would land in $B_{\bar \tau}$, and we must hence construct an oracle that moves $y$ outside it.
    Therefore define $r \coloneqq \norm{x}, u \coloneqq \frac x r, \mu \coloneqq \norm{\nabla F(x)}, \mu' \coloneqq \frac{L_\nu r^{\nu - 1}}{2^{2-\nu}} \geq \mu$ and
    \begin{equation*}
        g_\pm \coloneqq \frac{r \pm \bar \tau}{\eta} \geq 0, \qquad
        \delta_\pm  \coloneqq \abs{\mu - g_\pm}, \qquad
        \Delta \coloneqq g_+ - g_- = \frac{2 \bar \tau}{\eta}.
    \end{equation*}
    Note that by our case assumption, we have $g_- < \mu < g_+$. We next define the probability $\delta \coloneqq \frac{\delta_-}{\Delta}$ and gradient oracle
    \begin{equation*}
        \nf{x, \xi} \coloneqq \begin{cases}
            g_- u, & \xi \geq \delta\\
            g_+ u, & \xi \leq \delta,
        \end{cases}
    \end{equation*}
    where $\xi \sim \operatorname{Unif}([0,1])$. By construction we have that $\norm{x - \eta \nf{x, \xi}} = \abs{r - \pare{r \pm \bar \tau}} \norm{u} = \bar \tau$ and hence it only remains to show that $\nf{x, \xi}$ satisfies the noise assumptions. Therefore first note that this oracle is unbiased by
    \begin{align*}
        \Exp{\nf{x, \xi}}
        &= \pare{1-\delta} g_- u + \delta g_+ u
        = \frac{\pare{g_+ - \mu} g_- + \pare{\mu - g_-} g_+}{\Delta} u
        = \frac{\mu\pare{g_+ - g_-}}{\Delta} u = \nabla F(x),
    \end{align*}
    where we used $\delta = \frac{\mu - g_-}{\Delta}, 1-\delta = \frac{g_+ - \mu}{\Delta}$ and the definition of $\Delta$. Next we check the bounded $p$-th central moment property. Therefore we first calculate
    \begin{equation*}
        \Exp{\norm{\nf{x, \xi} - \nabla F(x)}^p}
        = \pare{1 - \delta} \delta_-^p + \delta \delta_+^p
        = \Delta^p \pare{\pare{1- \delta}\delta^p + \delta \pare{1 - \delta}^p}.
    \end{equation*}
    Now let $s \coloneqq \min \set{\delta, 1 - \delta}$ and note that $\pare{1- \delta}\delta^p + \delta \pare{1 - \delta}^p \leq \delta^p + \delta \leq 2 \delta$. Using a symmetric argument for $1-\delta$, we get that $\pare{1- \delta}\delta^p + \delta \pare{1 - \delta}^p \leq 2 s$. Next, by definition, we have $s \leq \nicefrac 1 2$ and
    \begin{equation*}
        \delta
        = \frac{\mu - g_-} \Delta
        \leq \frac{\mu' - g_-} \Delta
        = \frac{\eta L_\nu r^{\nu-1}}{ 2^{2-\nu} 2 \bar \tau} - \frac{r - \bar \tau}{2 \bar \tau}
        = A t^{\nu-1} - \frac{t - 1}{2},
    \end{equation*}
    where  $A \coloneqq \frac{\eta L_\nu}{2\pare{2\bar \tau}^{2-\nu}}$, and $t \coloneqq \frac r {\bar \tau} \geq 1$. Noting that, whenever $A \leq \nicefrac 1 2$,
    \begin{equation*}
         A t^{\nu-1} - \frac{t - 1}{2}
         = A + A \pare{t^{\nu-1} - 1} - \frac{t - 1}{2}
         \leq A + \frac{t^{\nu-1} - t}{2}
         \leq A
    \end{equation*}
    yields $s \leq A$ and hence
    \begin{align*}
        \Exp{\norm{\nf{x, \xi} - \nabla F(x)}^p}
        &\leq \Delta^p 2 s
        \leq \pare{\frac{2 \bar \tau}{\eta}}^p \frac{\eta L_\nu}{\pare{2\bar \tau}^{2-\nu}}
        = \pare{2 \bar \tau}^{p + \nu - 2} \frac {L_\nu} {\eta^{p-1}}
        \\
        & \leq \pare{2 \tau}^{p + \nu - 2} \frac {L_\nu} {\eta^{p-1}}
        = \frac{\sigma^p}{2^p}.
    \end{align*}
    By Jensen's inequality we have $\norm{\frac{v+w}{2}}^p \leq \frac{\norm{v}^p + \norm{w}^p}{2}$ for all $v, w \in \R^d$, and hence $\norm{v + w}^p \leq 2^{p-1}\pare{\norm v^p + \norm w^p}$. Finally we use this fact to get
    \begin{equation*}
        \Exp{\norm{\nabla f(x, \xi)}^p}
        \leq 2^{p-1} \pare{\norm{\nabla F(x)}^p + \Exp{\norm{\nabla f(x, \xi) - \nabla F(x)}^p} }
        \leq \frac {G^p} 2 + \frac{\sigma^p}{2}
        \leq G^p.
    \end{equation*}
    Hence the gradient oracle satsifies the $p$-BM and $p$-BCM assumption, completing the proof.
\end{proof}

By iteratively applying \Cref{prop:nonconvex.sgd_lb.stoch_large_steps}, we get the lower bound
\begin{equation}\label{eq:nonconvex.lower_bound.SGD_lower_bound.norm_xT_lb}
    \min_{t \in [T]} \norm{x_t} \geq \min\set{\norm{x_1}, \min_{t \in [T]} \tau_t}, \qquad \text{where} \qquad \tau_t \coloneqq \frac 1 2 \pare{\frac{\eta_t^{p-1}\sigma^p}{2^pL_\nu}}^{\frac 1 {p + \nu - 2}}
\end{equation}
on the iterates. Translating it to a gradient norm lower-bound, and combining it with \Cref{prop:nonconvex.sgd_lb.deterministic_small_steps} yields \Cref{thm:nonconvex.lower_bound.SGD_lower_bound}. The formal argument is carried in \Cref{sec:app.missing_proofs}.

\paragraph{Discussion} In the special case $p = \nu$, applying \Cref{thm:nonconvex.lower_bound.SGD_lower_bound} with $\sigma = G$ establishes a sample complexity lower bound of
\begin{equation}\label{eq:nonconvex.lower_bound.sgd_polynomial_stepsize_sample_complexity_lower_bound}
    T = \Omega\pare{
        \pare{\frac{\Delta_1}{\eta \eps^2}}^{\frac{1}{1-r}}
        + \eta^{\frac 1 r} \pare{\frac{L_p G^p}{\eps^2}}^{\frac 1 {r(p-1)}}
    },
\end{equation}
confirming the tightness of our result in all parameters for unconstrained problems. Notably, our analysis demonstrates that $\eta_t \propto t^{-\frac 1 p}$ is the uniquely optimal polynomial decay. For constrained problems, the iterates of our lower-bound stay within a domain of diameter $\nicefrac{\Delta_1}{\varepsilon}$, and hence confirm tightness of our upper-bound result whenever $D_\cX \geq \nicefrac{\Delta_1}{\eps}$.

For domains with smaller diameter, the situation is more complicated, even in the classical $p = \nu = 2$ setting. There, our lower-bound --- after choosing $\Delta_1 = \eps D_{\cX}$ to guarantee $\norm{x_t - x*} \leq D_{\cX}$ --- reduces to $\Omega\pare{\eps^{-3}}$, while the upper-bound is of order $\Oc\pare{\eps^{-4}}$. For convex functions, this gap can be closed from above \citep{projectedSGDconstrained2024lan}, for non-convex functions this question is still open even in this classical setting to the best of our knowledge. An overview can be found in \Cref{tab:nonconvex.sgd_lb.bounded_domain_overview}.

\newcommand{\tightColor}{\cellcolor{green!10}}
\newcommand{\looseColor}{\cellcolor{red!10}}
\begin{table}[tbp]
    \centering
    \begin{tabular}{cccc}
        \toprule
        Init.\ Assumption & Convex & Lower-Bound & Upper-Bound \\ \midrule
        \multirow{2}{*}{$F(x_1) - F^* \leq \Delta_1$} & \xmark & \tightColor & \tightColor$\Oc\pare{\eps^{-4}}$ \citep{Ghadimi2013} \\
        & \cmark & \tightColor\multirow{-2}{*}{ $\Omega \pare{\eps^{-4}}$ \citep{drori2020complexity}} & \tightColor$\Oc \pare{\eps^{-4}}$ \citep{drori2020complexity} \\ \midrule
        \multirow{2}{*}{$\operatorname{diam}(\cX) \leq D_{\cX}$} & \xmark & \looseColor$\Omega \pare{\eps^{-3}}$ (\Cref{thm:nonconvex.lower_bound.SGD_lower_bound}) & \looseColor$\Oc \pare{\eps^{-4}}$ \citep{Ghadimi2016} \\
        & \cmark & \tightColor$\Omega \pare{\eps^{-3}}$ (\Cref{thm:nonconvex.lower_bound.SGD_lower_bound}) & \tightColor$\Oc \pare{\eps^{-3}}$ \citep{projectedSGDconstrained2024lan}\\
        \bottomrule
    \end{tabular}
    \caption{Gradient oracle complexity upper- and lower-bounds for reaching an $\eps$-stationary point in the classical stochastic setting ($p = \nu = 2$) using (mini-batch) \algname{SGD} with polynomially decaying stepsizes. The bounded-domain lower-bounds follow by choosing $\Delta_1 = \eps D_{\cX}$ in \eqref{eq:nonconvex.lower_bound.sgd_polynomial_stepsize_sample_complexity_lower_bound}, guaranteeing that the iterates stay in a domain $\cX$ with $\operatorname{diam}(\cX) = \nicefrac {\Delta_1} \eps = D_{\cX}$. To the best of our knowledge, the optimal sample complexity for the non-convex, bounded domain setting remains open even in this classical setting.}
    \label{tab:nonconvex.sgd_lb.bounded_domain_overview}
\end{table}

Additionally note that, while \Cref{thm:nonconvex.lower_bound.SGD_lower_bound} is derived for polynomial stepsizes for simplicity, it also holds for arbitrary stepsizes, scaling with $\sum_{t=1}^T \eta_t$ and $\min_{t \in [T]} \eta_t$ respectively.

\paragraph{Comparison to prior work}
The most closely related prior result established a lower bound of $\Omega(\min\set{L^2_2\Delta_1^2,\sigma^4} \eps^{-4})$ iterations for SGD under standard smoothness and bounded variance assumptions \citep{drori2020complexity}. We extend this prior complexity bound to the broader classes of Hölder-smooth ($\nu \in (1,2]$) and heavy-tailed gradient noise distributions ($p \in (1,2]$), recovering existing results in the classical $p=\nu=2$ case as special instances. Furthermore, their construction crucially depends on a restrictive high-dimensionality assumption $d \geq \tilde{\Omega}(\sigma^2 \Delta_1 \eps^{-4})$, limiting applicability. In contrast, our construction removes this dimensionality constraint entirely through a carefully designed gradient oracle. Finally, our analysis encompasses arbitrary stepsize sequences, whereas parts of the previous result requires specific schedules \citep[Proposition 2]{drori2020complexity}.

Beyond SGD, the lower bounds for general first-order methods to reach an $\eps$-stationary point have been studied extensively. In the classical ($p = \nu = 2$) setting, $\Omega(\varepsilon^{-2})$ samples are required for \emph{convex} functions \citep{makingGradientSmallStochasticConvex2019Foster}. When $p \in (1, 2]$ and $\nu = 2$, \citep{zhang2020adaptive} establish a $\Omega(\varepsilon^{-(3p-2)/(p-1)})$ lower-bound for non-convex functions, which we recover as special case when $\nu = 2$ for \algname{SGD}. To the best of our knowledge, Hölder-smoothness has only been addressed in constrained convex optimization \citep{guzman2015hoeldersmoothlowerbound,bai2025tight}, with bounds expressed in terms of suboptimality rather than gradient stationarity. Consequently, our results uniquely address the combination of Hölder-smoothness and heavy-tailed noise simultaneously, offering lower bounds applicable to both convex and nonconvex settings.

\subsection[Upper Bound for SGD under Standard Smoothness]{Upper Bound for \algname{SGD} under Standard Smoothness}\label{subsec:mini_batch_SGD}
Previously, we studied vanilla \algname{SGD} under Hölder smoothness and derived tight convergence rates along with a lower bound construction. While Hölder smoothness is weaker than standard smoothness ($\nu = 2$) when the set $\cX$ is bounded, our goal in this section is to use standard smoothness directly. We first recall the definition of standard smoothness.

\begin{assumption}[Smoothness]\label{ass:smooth_standard}
  \Cref{ass:smooth} holds with $\nu = 2$, i.e., there exist $\ell$ and $L$ such that
	$
	- \frac{\ell}{2} \norm{x-y}^{2} \leq F(x) - F(y) - \langle \nabla F(y), x - y \rangle \leq  \frac{L}{2} \norm{x-y}^{2}$ for all $x, y \in \cX. $
\end{assumption}
In this subsection, we consider $\nu = 2,$ and will omit the subscript in constants $\ell \eqdef \ell_2$, $L\eqdef L_2$, and the superscript in $F_{\nicefrac{1}{\rho}}(\cdot) := F_{\nicefrac{1}{\rho}}^{2}(\cdot),$ $\hat x := \hat x^{2} .$ That is, we have $
F_{\nicefrac{1}{\rho}}(x) := \min_{y\in \cX} \sb{ F(y) +  \frac{\rho}{2} \norm{y - x}^{2} }$, $\hat x := \argmin_{y\in \cX} \sb{ F(y) + \frac{\rho}{2} \norm{y - x}^{2}  } $ for any $x \in \cX .$ We also refer to \Cref{def:FBE_GSM} for corresponding definitions of  $\mathcal{D}_{\rho}^{2}(x)$ and $\mathcal{S}_{\rho}^{2}(x)$, which coincide for $\nu = 2.$ We also use the refined \textit{central} moment noise assumption replacing \hyperref[assum:pBM]{(p-BM)}.

\begin{assumption}\label{assum:pBCM}
	Let $F(\cdot)$ be differentiable on $\cX$. We have access to stochastic gradients with $\Exp{\nabla f(x, \xi)} = \nabla F(x)$ and there exists $p \in (1, 2]$ such that
	$$
	p\text{-BCM} \qquad \Exp{\norm{\nabla f(x, \xi) - \nabla F(x)}^p} \leq \sigma^p \qquad \text{for all } x \in \cX .
	$$
\end{assumption}
For a step-size sequence $\cb{\eta_t}_{t\geq1}$, with the convention $\eta_0=\eta_1$, our analysis is based on the Lyapunov function used in \cite{fatkhullin2024taming}:
\begin{equation*}
	\lambda_t\eqdef F_{\nicefrac{1}{\rho}}(x_t)-F^*
	+\eta_{t-1}\rho\rb{F(x_t)-F^*}.
\end{equation*}
Now we are ready to present the convergence guarantee for \algname{SGD} under standard smoothness.

\begin{tcolorbox}[colback=white,colframe=mygreen!75!black]
\begin{theorem}
	\label{the:minibatch_SGD}
{\color{black}{	Let Assumptions~\ref{ass:lower_bounded_constr}, \ref{ass:smooth_standard}, and \ref{assum:pBCM} hold. Denote $\Delta_1 \eqdef F(x_1)-F^*$ and set $\rho \eqdef 2(L+2\ell)$. Suppose $\Delta_1>0$ and \algname{SGD} is run for $T\geq 1$ iterations with the constant step-size
	\begin{equation}\label{eq:SGD_standard_smooth_stepsize}
		\eta_t \equiv \eta \eqdef
		\min\cb{\frac{1}{\rho},\,
		\frac{\sqrt{\Delta_1}}{\sigma\sqrt{\rho}\,T^{1/p}}},
	\end{equation}
	where the second quantity is interpreted as $+\infty$ if $\sigma=0$. Then
	\begin{equation}\label{eq:SGD_standard_smooth_rate}
		\frac{1}{T}\sum_{t=1}^T
		\Exp{\rb{\mathcal S_{\rho+L}^2(x_t)}^{\frac p2}}
		\leq 24\left[
		\pare{\frac{\rho\Delta_1}{T}}^{\frac p2}
		+\frac{\pare{\sigma\sqrt{\rho\Delta_1}}^{\frac p2}}{T^{\frac{p-1}{2}}}
		+\frac{\sigma^p}{T^{p-1}}
		\right].
	\end{equation}
    }}
\end{theorem}
\end{tcolorbox}
\begin{proof}
Let $q\eqdef p/2\in(1/2,1]$, define $D_t\eqdef\mathcal D_{\rho+L}^2(x_t)=\mathcal S_{\rho+L}^2(x_t)$, and introduce
\begin{equation*}
	\psi_t\eqdef\nabla f(x_t,\xi_t)-\nabla F(x_t),\qquad
	h_t\eqdef\hat x_t-x_t,\qquad
	s_t\eqdef x_t-x_{t+1}.
\end{equation*}
{\color{black}{Specializing Equation~(15) of \cite{fatkhullin2024taming}}\allowbreak} to the Euclidean setup, and using the lower curvature constant $\ell$ and the upper curvature constant $L$, gives
\begin{align}\label{eq:SGD_standard_smooth_one_step}
	\lambda_{t+1}\leq{}&\lambda_t-a\eta D_t
	+\rho\eta\langle\psi_t,h_t\rangle-b\eta\norm{h_t}^2\notag\\
	&\qquad +\rho\eta\langle\psi_t,s_t\rangle
	-\frac{\rho(1-\eta L)}{2}\norm{s_t}^2,
\end{align}
where
$
	a\eqdef\frac{\rho}{2(\rho+L)}\geq\frac13,
	\,
	b\eqdef\frac{\rho(\rho-\ell)}2\geq\frac{3\rho^2}{8}.
$
The inequalities follow from $\rho\geq2L$ and $\rho\geq4\ell$.

The last line of \eqref{eq:SGD_standard_smooth_one_step} cannot be integrated directly because $\Exp{\norm{\psi_t}^2}$ may be infinite. Instead, since $\eta L\leq L/\rho\leq1/2$, we retain the following pathwise bound:
\begin{equation}\label{eq:SGD_standard_smooth_B_bound}
	\rho\eta\langle\psi_t,s_t\rangle
	-\frac{\rho(1-\eta L)}2\norm{s_t}^2
	\leq\frac{\rho\eta^2}{2(1-\eta L)}\norm{\psi_t}^2
	\leq\rho\eta^2\norm{\psi_t}^2.
\end{equation}
Telescoping \eqref{eq:SGD_standard_smooth_one_step}, using \eqref{eq:SGD_standard_smooth_B_bound} and $\lambda_{T+1}\geq0$, yields
\begin{equation}\label{eq:SGD_standard_smooth_telescope}
	a\eta\sum_{t=1}^T D_t+bY_T
	\leq\lambda_1+M_T+R_T,
\end{equation}
where
$
	Y_T\eqdef\eta\sum_{t=1}^T\norm{h_t}^2,\,
	M_T\eqdef\rho\eta\sum_{t=1}^T\langle\psi_t,h_t\rangle,\,
	R_T\eqdef\rho\eta^2\sum_{t=1}^T\norm{\psi_t}^2.
$

Next, our goal is to keep $M_T$ as a single martingale transform, rather than applying Young's inequality to each of its increments. Since $M_T\leq|M_T|$, \eqref{eq:SGD_standard_smooth_telescope}, using $q\leq1$ and subadditivity of $u\mapsto u^q$ on $\R_+$ yields
\begin{align}\label{eq:SGD_standard_smooth_power}
	2^{q-1}\left[
	a^q\Exp{\rb{\eta\sum_{t=1}^T D_t}^q}
	+b^q\Exp{Y_T^q}
	\right]
	\leq\lambda_1^q+\Exp{|M_T|^q}+\Exp{R_T^q}.
\end{align}
The last term is finite under \hyperref[assum:pBCM]{(p-BCM)}, since $2q=p$ and
\begin{equation}\label{eq:SGD_standard_smooth_R}
	\Exp{R_T^q}
	\leq\rho^q\eta^{2q}\sum_{t=1}^T\Exp{\norm{\psi_t}^{2q}}
	\leq\rho^q\eta^p T\sigma^p.
\end{equation}

It remains to control $M_T$. With respect to the natural filtration, $\cb{\rho\eta\langle\psi_t,h_t\rangle}_{t=1}^T$ is a martingale difference sequence. Hence, \Cref{lem:app.von_Bahr_and_Essen} and \hyperref[assum:pBCM]{(p-BCM)} imply
\begin{align}
	\Exp{|M_T|^p}
	&\leq2(\rho\eta)^p\sum_{t=1}^T
	\Exp{\abs{\langle\psi_t,h_t\rangle}^p}
	\leq2(\rho\eta)^p\sigma^p\sum_{t=1}^T\Exp{\norm{h_t}^p}
	\leq2\rho^p\sigma^p\eta^qT^{1-q}\Exp{Y_T^q},
	\label{eq:SGD_standard_smooth_M_p}
\end{align}
where the last inequality follows from
$
	\sum_{t=1}^T\norm{h_t}^p
	\leq T^{1-q}\pare{\sum_{t=1}^T\norm{h_t}^2}^{q}
	=T^{1-q}\eta^{-q}Y_T^q.
$
By monotonicity of $L^r$-norms and Young's inequality,
\begin{align}
	\Exp{|M_T|^q}
	&\leq\Exp{|M_T|^p}^{1/2} \leq\sqrt{2}\rho^q\sigma^q\eta^{q/2}T^{(1-q)/2}
	\Exp{Y_T^q}^{1/2}\notag\\
	&\leq\frac{2^{q-1}b^q}{2}\Exp{Y_T^q}
	+\frac{\rho^p}{2^{q-1}b^q}\sigma^p\eta^qT^{1-q}.
	\label{eq:SGD_standard_smooth_M_q}
\end{align}
The first term on the last line is absorbed by the left-hand side of \eqref{eq:SGD_standard_smooth_power}. After substituting \eqref{eq:SGD_standard_smooth_R} and \eqref{eq:SGD_standard_smooth_M_q} into \eqref{eq:SGD_standard_smooth_power} and dividing by $(\eta T)^q$, the coefficients of the first and third terms below are bounded by $1/(2^{q-1}a^q)\leq2^{1-q}3^q\leq3$, whereas the coefficient of the second term is bounded by $\rho^p/(2^{2q-2}a^qb^q)\leq(8/3)^q/(2^{2q-2}3^{-q})=2^{q+2}\leq8$. Here we used $a\geq1/3$, $b\geq3\rho^2/8$, and $p=2q$. Thus, $C_0=8$ is valid, and we obtain
\begin{equation}\label{eq:SGD_standard_smooth_master}
	\frac1T\sum_{t=1}^T\Exp{D_t^q}
	\leq\Exp{\pare{\frac1T\sum_{t=1}^T D_t}^{q}}
	\leq 8\left[
	\pare{\frac{\lambda_1}{\eta T}}^q
	+\frac{\sigma^p}{T^{p-1}}
	+\sigma^p(\rho\eta)^qT^{1-q}
	\right].
\end{equation}
The first inequality above follows from the concavity of $u\mapsto u^q$.

Finally, $F_{\nicefrac1\rho}(x_1)-F^*\leq\Delta_1$ and $\rho\eta\leq1$, so $\lambda_1\leq2\Delta_1$. Moreover,
\begin{equation*}
	\frac1\eta
	\leq\rho+\frac{\sigma\sqrt{\rho}}{\sqrt{\Delta_1}}T^{1/p},
	\qquad
	\eta\leq\frac{\sqrt{\Delta_1}}{\sigma\sqrt{\rho}}T^{-1/p}.
\end{equation*}
Consequently,
\begin{align*}
	\pare{\frac{\lambda_1}{\eta T}}^q
	&\leq
	\pare{\frac{2\rho\Delta_1}{T}}^q
	+\frac{\pare{2\sigma\sqrt{\rho\Delta_1}}^q}{T^{(p-1)/2}},\\
	\sigma^p(\rho\eta)^qT^{1-q}
	&\leq
	\frac{\pare{\sigma\sqrt{\rho\Delta_1}}^q}{T^{(p-1)/2}}.
\end{align*}
Since $2^q\leq2$, substituting these bounds into \eqref{eq:SGD_standard_smooth_master} gives
\begin{equation*}
	\frac1T\sum_{t=1}^T\Exp{D_t^q}
	\leq24\left[
	\pare{\frac{\rho\Delta_1}{T}}^q
	+\frac{\pare{\sigma\sqrt{\rho\Delta_1}}^q}{T^{(p-1)/2}}
	+\frac{\sigma^p}{T^{p-1}}
	\right],
\end{equation*}
which proves \eqref{eq:SGD_standard_smooth_rate} with $C=24$.

\end{proof}

\paragraph{Discussion}
Let $\bar x_T$ be sampled uniformly from $\cb{x_t}_{t=1}^T$, independently of the stochastic gradients, and consider the unconstrained case $\cX=\R^d$, where $\mathcal S_{\rho+L}^2(x)=\norm{\nabla F(x)}^2$. Assume $\ell=L$, so that $\rho=6L$, and, for simplicity, $18L\Delta_1>\eps^2$, so the last, non-dominating term in \eqref{eq:SGD_standard_smooth_rate} does not affect the rate for reaching accuracy $\eps$. Then \eqref{eq:SGD_standard_smooth_rate} gives
\begin{equation}\label{eq:SGD_standard_smooth_unconstrained_rate}
	\Exp{\norm{\nabla F(\bar x_T)}^p}^{\frac2p}
	=\Oc\left(
	\frac{L\Delta_1}{T}
	+\frac{\sigma\sqrt{L\Delta_1}}{T^{\frac{p-1}{p}}}
	\right).
\end{equation}

For comparison, under $p$-Hölder smoothness and \hyperref[assum:pBM]{(p-BM)}, \Cref{thm:NC_SGD_upper}, with $\ell_p=L_p$ and the constant step-size $\eta=\Delta_1^{1/p}/\rb{G(\rho_pT)^{1/p}}$, where $\rho_p=6L_p/(p-1)$, gives
\begin{equation}\label{eq:SGD_holder_unconstrained_rate}
	\Exp{\norm{\nabla F(\bar x_T)}^2}
	=\Oc\left(
	\frac{G L_p^{1/p}\Delta_1^{\frac{p-1}{p}}}{T^{\frac{p-1}{p}}}
	\right).
\end{equation}
Thus, the slowest-decaying term in \eqref{eq:SGD_standard_smooth_unconstrained_rate} and the rate in \eqref{eq:SGD_holder_unconstrained_rate} both scale as $T^{-(p-1)/p}$. Their dependence on the other problem parameters is different: the corresponding standard-smoothness term is proportional to $\sigma\sqrt{L\Delta_1}$, whereas the $p$-Hölder-smoothness term is proportional to $G L_p^{1/p}\Delta_1^{(p-1)/p}$. The two bounds are not directly comparable because \hyperref[assum:pBCM]{(p-BCM)} controls the central moment by $\sigma$, while \hyperref[assum:pBM]{(p-BM)} controls the full stochastic gradient moment by $G$. Moreover, \eqref{eq:SGD_holder_unconstrained_rate} controls the stronger second-moment criterion, since $\Exp{\norm{\nabla F(\bar x_T)}^p}^{2/p}\leq\Exp{\norm{\nabla F(\bar x_T)}^2}$. For $p=2$, the convergence criteria in \eqref{eq:SGD_standard_smooth_unconstrained_rate} and \eqref{eq:SGD_holder_unconstrained_rate} coincide.

We next discuss the optimality of the standard-smoothness rate. Since $\bar x_T$
is sampled uniformly and independently from $\cb{x_t}_{t=1}^T$, Jensen's
inequality gives
\begin{equation*}
	\Exp{\pare{\frac1T\sum_{t=1}^T\norm{\nabla F(x_t)}}^p}
	\leq \frac1T\sum_{t=1}^T\Exp{\norm{\nabla F(x_t)}^p}
	=\Exp{\norm{\nabla F(\bar x_T)}^p}.
\end{equation*}
Consequently, for every $\delta\in(0,1)$, Markov's inequality and
\eqref{eq:SGD_standard_smooth_unconstrained_rate} imply that, with probability
at least $1-\delta$,
\begin{equation}\label{eq:SGD_standard_smooth_average_gradient_rate}
	\pare{\frac1T\sum_{t=1}^T\norm{\nabla F(x_t)}}^2
	=\Oc\left(
	\frac{L\Delta_1}{T} \delta^{-\frac{2}{p}}
	+\frac{\sigma\sqrt{L\Delta_1}}{T^{1-\frac{1}{p}}} \delta^{-\frac{2}{p}}
	\right).
\end{equation}

In the following subsection, we study the tightness of this upper bound in the
noise-dominated regime.

\subsection[Lower Bound for SGD under Standard Smoothness]{Lower Bound for \algname{SGD} under Standard Smoothness}\label{subsec:Standard_smoothness_lower_bound}

We prove the following lower bound for \algname{SGD}.

\begin{tcolorbox}[colback=white,colframe=mygreen!75!black]
\begin{theorem}[Standard-smooth lower bound for non-increasing step-sizes]
\label{thm:main_lower_bound_standard_smooth}
{\color{black}{Let $p\in(1,2]$, $T\geq4$, $\delta\in(0,1/8]$, $L>0$,
$\Delta_1\geq0$, and $\sigma\geq0$. For every deterministic step-size sequence
$\eta_1\geq\eta_2\geq\cdots\geq\eta_T>0$, there exist $x_1\in\R$, a
one-dimensional lower-bounded $L$-smooth function $F:\R\to\R$ with
$F(x_1)-\inf_{x\in\R}F(x)\leq\Delta_1$, and independent stochastic-gradient
oracle samples $\nabla f(x,\xi_t)$ satisfying
\hyperref[assum:pBCM]{(p-BCM)} with parameter $\sigma$, such that the iterates of \algname{SGD} satisfy, with probability at
least $\delta$,
\begin{equation}
\label{eq:SGD_standard_smooth_lower_bound_optimized}
    \pare{\frac1T\sum_{t=1}^T \norm{\nabla F(x_t)}}^2
    =
    \Omega\pare{
        \min\cb{
            L\Delta_1,\,
            \frac{\sigma\sqrt{L\Delta_1}}{T^{1-\frac1p}}
            \delta^{-\frac1p}
        }
    }.
\end{equation}
}}
\end{theorem}
\end{tcolorbox}

The lower-bound mechanism is illustrated in
\Cref{fig:standard_smooth_lower_bound_barrier}. On the rare event in
\Cref{lem:jump_stochastic_lower_bound_standard_smoothness}, the iterates remain
at zero until $x_\tau=0$, and the noise sample sends the next iterate to
$x_{\tau+1}=y$; the subsequent exact-gradient steps remain on the segment from
\Cref{lem:long-ramp}, where $\nabla F=-h$.

\begin{figure}[H]
	\centering
	\resizebox{0.88\textwidth}{!}{%
		\definecolor{LBInk}{HTML}{17201F}
\definecolor{LBSlate}{HTML}{4D5A58}
\definecolor{LBTeal}{HTML}{2A9D8F}
\definecolor{LBBlue}{RGB}{33,92,175}
\definecolor{LBCoral}{HTML}{C85B45}

\begin{tikzpicture}[
    x=8.4cm,
    y=1.78cm,
    >=Latex,
    every node/.style={font=\small,text=black}
]
    \pgfmathsetmacro{\lbL}{8}
    \pgfmathsetmacro{\lby}{1}
    \pgfmathsetmacro{\lbR}{2}
    \pgfmathsetmacro{\lbh}{1/8}
    \pgfmathsetmacro{\lbs}{(\lby-\lbh/\lbL)/2}
    \pgfmathsetmacro{\lbxR}{\lby+\lbh*\lbR}
    \pgfmathsetmacro{\lbcap}{\lbh/\lbL}
    \pgfmathsetmacro{\lbxF}{\lbxR+\lbcap}
    \pgfmathsetmacro{\lbFs}{0.5*\lbL*\lbs*\lbs}
    \pgfmathsetmacro{\lbFy}{\lbL*\lby*\lby/4-\lbh*\lby/2-\lbh*\lbh/(4*\lbL)}
    \pgfmathsetmacro{\lbFR}{\lbFy-\lbh*(\lbxR-\lby)}
    \pgfmathsetmacro{\lbFF}{\lbFR-\lbh*\lbh/(2*\lbL)}
    \pgfmathsetmacro{\lbD}{\lby-\lbs}

    \path[use as bounding box] (-0.12,-0.52) rectangle (1.40,2.48);

    \draw[->,LBSlate!75,line width=0.55pt] (-0.10,0) -- (1.37,0)
        node[below left=1pt] {$x$};
    \draw[->,LBSlate!75,line width=0.55pt] (-0.065,-0.08) -- (-0.065,2.23);
    \node[anchor=east] at (-0.078,2.10) {$F(x)$};

    \draw[LBTeal!16,line width=5.0pt,line cap=round]
        (\lby,\lbFy) -- (\lbxR,\lbFR);

    \draw[LBInk,line width=1.15pt,line cap=round,line join=round]
        (-0.08,0) -- (0,0)
        .. controls ({\lbs/3},0)
          and ({2*\lbs/3},{\lbFs-\lbL*\lbs*\lbs/3})
          .. (\lbs,\lbFs)
        .. controls ({\lbs+\lbD/3},{\lbFs+\lbD*\lbL*\lbs/3})
          and ({\lby-\lbD/3},{\lbFy+\lbD*\lbh/3})
          .. (\lby,\lbFy)
        -- (\lbxR,\lbFR)
        .. controls ({\lbxR+\lbcap/3},{\lbFR-\lbcap*\lbh/3})
          and ({\lbxF-\lbcap/3},\lbFF)
          .. (\lbxF,\lbFF)
        -- (1.34,\lbFF);
    \node at (0.60,0.95) {barrier};

    \fill[LBBlue] (0,0) circle (1.8pt);
    \node[below=3pt] at (0,0) {$x_1=0$};

    \fill[LBCoral] (\lby,\lbFy) circle (1.8pt);
    \draw[densely dashed,LBCoral!70] (\lby,\lbFy) -- (\lby,0);
    \node[anchor=north east] at (0.99,-0.02) {$x_{\tau+1}=y$};
    \draw[->,LBCoral,line width=1.25pt]
        (0,0.04) .. controls (0.23,2.26) and (0.72,2.25)
        .. (0.992,2.00);
    \node[anchor=west] at (0.18,2.24)
        {first nonzero noise $\xi_\tau=-y/\eta_\tau$};

    \foreach \k in {0,1,2,3}{
        \pgfmathsetmacro{\lbxa}{1+\k/32}
        \pgfmathsetmacro{\lbxb}{1+(\k+1)/32}
        \pgfmathsetmacro{\lbya}{\lbFy-\lbh*(\lbxa-1)}
        \pgfmathsetmacro{\lbyb}{\lbFy-\lbh*(\lbxb-1)}
        \draw[LBBlue,line width=1.05pt,-{Latex[length=1.7mm]}]
            (\lbxa,\lbya) -- (\lbxb,\lbyb);
    }
    \node[anchor=south] at (1.095,2.03)
        {exact gradient steps};

    \node at ({(\lby+\lbxR)/2},1.61)
        {$\nabla F=-h$};
    \draw[<->,LBSlate!75,line width=0.55pt]
        (\lby,-0.33) -- (\lbxR,-0.33)
        node[midway,below=2pt] {width $h\sum_{t=1}^{T}\eta_t$};
\end{tikzpicture}%
	}
	\caption{Rare noise moves \algname{SGD} over the barrier to a
	constant-gradient segment, where exact-gradient steps keep
	$\norm{\nabla F(x_t)}=h$ for $t=\tau+1,\ldots,T$.}
	\label{fig:standard_smooth_lower_bound_barrier}
\end{figure}

Before proving the theorem, we state one proposition and two auxiliary lemmas. The proof of the
first is deferred to \Cref{sec:app.missing_proofs}; the other two are proved
below. The first proposition is similar to
\Cref{prop:nonconvex.sgd_lb.deterministic_small_steps}, and serves to ``punish'' small cumulative step-sizes.

\begin{proposition}
\label{prop:deterministic_lower_bound_standard_smoothness}
Let $T\geq1$, $L>0$, $\Delta_1\geq0$, and let
$\eta_1,\ldots,\eta_T>0$ with $S_T\eqdef\sum_{t=1}^T\eta_t$. There exist
$x_1\in\R$, a one-dimensional lower-bounded $L$-smooth function $F:\R\to\R$,
and a deterministic oracle $\nabla f(x,\xi)=\nabla F(x)$ such that
$F(x_1)-\inf_{x\in\R}F(x)\leq\Delta_1$ and
\[
    \pare{\frac1T\sum_{t=1}^T\norm{\nabla F(x_t)}}^2
    \geq
    \frac13\min\cb{L\Delta_1,\Delta_1/S_T} .
\]
\end{proposition}

The following lemma constructs an aligned rare long jump and quantifies the probability that it occurs during the first half of the iterations.

\begin{lemma}
\label{lem:jump_stochastic_lower_bound_standard_smoothness}
Let $p\in(1,2]$, $T\geq4$, $m=\lfloor T/2\rfloor$,
$\delta\in(0,1/8]$, $\sigma>0$, and let $\eta_1,\ldots,\eta_m>0$. Set
\[
    A_{p,m}\eqdef\pare{\sum_{t=1}^m\eta_t^p}^{1/p},
    \qquad
    y\eqdef\sigma A_{p,m}(4\delta)^{-1/p}.
\]
There are independent centered random variables $\xi_1,\ldots,\xi_T$ such
that $\Exp{\norm{\xi_t}^p}\leq\sigma^p$ for every $t$ and, with probability at
least $\delta$, the first nonzero sample among $\xi_1,\ldots,\xi_m$ occurs at
some time $\tau\leq m$ and satisfies
\[
    \eta_\tau(-\xi_\tau)=y.
\]
\end{lemma}

\begin{proof}
For $t\leq m$, define
\[
    q_t\eqdef\frac{2\delta\eta_t^p}{A_{p,m}^p},
    \qquad
    \xi_t=
    \begin{cases}
        -y/\eta_t, & \text{with probability }q_t,\\
        y/\eta_t, & \text{with probability }q_t,\\
        0, & \text{with probability }1-2q_t .
    \end{cases}
\]
For $t>m$, set $\xi_t=0$. Since $\sum_{t=1}^m q_t=2\delta$ and
$\delta\leq1/8$, each $q_t\leq1/4$, so the distributions are well-defined.
They are symmetric, hence centered, and, for $t\leq m$,
\[
    \Exp{\norm{\xi_t}^p}
    =
    2q_t\pare{\frac{y}{\eta_t}}^p
    =
    4\delta\frac{y^p}{A_{p,m}^p}
    =
    \sigma^p .
\]
The same moment bound is trivial for $t>m$.

Let $E$ be the event that the first nonzero sample among
$\xi_1,\ldots,\xi_m$ is negative. Independence gives the probability of this event
\[
    \Prob{E}
    =
    \sum_{\tau=1}^m q_\tau\prod_{s<\tau}(1-2q_s).
\]
For every $\tau$,
\[
    \prod_{s<\tau}(1-2q_s)
    \geq
    1-2\sum_{s<\tau}q_s
    \geq
    1-4\delta,
\]
where the first inequality is the union bound applied to the complementary
event. Therefore
\[
    \Prob{E}
    \geq
    (1-4\delta)\sum_{\tau=1}^m q_\tau
    =
    2\delta(1-4\delta)
    \geq
    \delta .
\]
On $E$, if $\tau$ is the first nonzero time, then
$\xi_\tau=-y/\eta_\tau$, and hence $\eta_\tau(-\xi_\tau)=y$.
\end{proof}

The next lemma is the main non-convex hard instance construction. It builds a one-sided smoothed finite potential barrier, and quantifies the gradient accumulation of \algname{SGD} after it jumps over the barrier and slides along the linear region.

\begin{lemma}
\label{lem:long-ramp}
Let $T\geq4$, $\{\eta_t\}_{t=1}^{T}>0$, $L>0$, and $y>0$. Let
$R\geq\max\cb{4/L,\sum_{t=1}^T\eta_t}$. Let $h>0$ satisfy
\[
    h^2\leq2^{-8}\frac{Ly^2}{R}.
\]
Then there is a one-dimensional lower-bounded $L$-smooth function $F:\R\to\R$
such that $F(0)-\inf_{x\in\R}F(x)=0$, $\nabla F(x)=0$ for all $x\leq0$, and
\[
    \nabla F(x)=-h
    \qquad
    \text{for all }x\in[y,y+hR].
\]
Moreover, if an iterate satisfies $x_{\tau+1}=y$ for some
$\tau\leq\lfloor T/2\rfloor$ and all subsequent updates are exact gradient
steps, then
\[
    \pare{\frac1T\sum_{t=1}^T\norm{\nabla F(x_t)}}^2
    \geq
    \frac{h^2}{4}.
\]
\end{lemma}

\begin{proof}
Put $s\eqdef\frac12\pare{y-h/L}$. The assumptions imply
\[
    \frac{h^2}{L^2y^2}
    \leq
    \frac{2^{-8}}{LR}
    \leq
    2^{-10},
\]
so $h\leq Ly/32$ and, in particular, $s>0$.
Define
\[
\nabla F(x)=
\begin{cases}
    0, & x\leq0,\\
    Lx, & 0<x\leq s,\\
    L(y-x)-h, & s<x\leq y,\\
    -h, & y<x\leq y+hR,\\
    -h+L(x-y-hR), & y+hR<x\leq y+hR+h/L,\\
    0, & x>y+hR+h/L .
\end{cases}
\]
At $x=s$, the two middle definitions agree because
$Ls=L(y-s)-h$.  Thus $\nabla F$ is continuous and $L$-Lipschitz.  Let
$F(0)=0$ and define $F$ by integration.  We now verify that $F$ is lower
bounded by $0$.

First,
\[
\begin{aligned}
    F(y)-F(0)
    &=
    \int_0^s Lx\,dx+\int_s^y \pare{L(y-x)-h}\,dx  \\
    &=
    \frac{Ly^2}{4}-\frac{hy}{2}-\frac{h^2}{4L}
    \geq
    \frac{7}{64}Ly^2,
\end{aligned}
\]
where the last inequality uses $h\leq Ly/4$.  The total decrease after $y$ is
\[
    h^2R+\frac{h^2}{2L}
    \leq
    2^{-8}Ly^2+2^{-11}Ly^2
    <
    \frac{7}{64}Ly^2 .
\]
Therefore the decrease after $y$ cannot push the function below $F(0)=0$.
Before $y$, the only region where $\nabla F$ is negative is the final segment
of $(s,y]$, and this segment ends at the nonnegative value $F(y)$.  Hence
$F(x)\geq0$ for every $x\in\R$, and $F(0)-\inf_{x\in\R}F(x)=0$.

Now suppose $x_{\tau+1}=y$ and $\tau\leq\lfloor T/2\rfloor$.  Since
$\nabla F=-h$ on $[y,y+hR]$, exact gradient steps give, by induction,
\[
    x_{\tau+j}=y+h\sum_{i=\tau+1}^{\tau+j-1}\eta_i,
    \qquad
    j=1,\ldots,T-\tau.
\]
For these indices,
$\sum_{i=\tau+1}^{\tau+j-1}\eta_i\leq\sum_{i=1}^T\eta_i\leq R$, so all the
displayed points lie in $[y,y+hR]$ and have gradient norm $h$. There are
$T-\tau\geq T-\lfloor T/2\rfloor\geq T/2$ such post-update points among
$x_1,\ldots,x_T$.
Therefore
\[
    \pare{\frac1T\sum_{t=1}^T\norm{\nabla F(x_t)}}^2
    \geq
    \pare{\frac{T-\tau}{T}h}^2
    \geq
    \frac{h^2}{4}.
\]
\end{proof}

\begin{proof}[Proof of~\Cref{thm:main_lower_bound_standard_smooth}]
If $\Delta_1=0$, take $F\equiv0$ and the deterministic zero oracle. Hence
assume $\Delta_1>0$.
Set
\[
    S_T\eqdef\sum_{t=1}^T\eta_t,
    \qquad
    m\eqdef\lfloor T/2\rfloor,
    \qquad
    A_{p,m}\eqdef\pare{\sum_{t=1}^m\eta_t^p}^{1/p},
\]
and let $c_{\mathrm{gen}}(p)\eqdef 2^{-10}4^{-2/p}$. We first prove the
schedule-dependent estimate
\begin{equation*}
    \pare{\frac1T\sum_{t=1}^T \norm{\nabla F(x_t)}}^2
    \geq
    c_{\mathrm{gen}}(p)\min\cb{
        L\Delta_1,\,
        \max\cb{
            \frac{\Delta_1}{S_T},\,
            \frac{L\sigma^2A_{p,m}^2}{S_T}\delta^{-\frac2p}
        }
    }.
\end{equation*}
The deterministic term in this estimate follows from
\Cref{prop:deterministic_lower_bound_standard_smoothness}. We now prove the
stochastic term.

If $S_T<4/L$, then
\Cref{prop:deterministic_lower_bound_standard_smoothness} gives
\[
    \pare{\frac1T\sum_{t=1}^T\norm{\nabla F(x_t)}}^2
    \geq
    \frac13\min\cb{L\Delta_1,\Delta_1/S_T}
    \geq
    \frac{L\Delta_1}{12}
    \geq
    c_{\mathrm{gen}}(p)L\Delta_1,
\]
which proves the schedule-dependent estimate in this case.
Hence assume $S_T\geq4/L$.
If $\sigma=0$, the stochastic term is zero, so
\Cref{prop:deterministic_lower_bound_standard_smoothness} and
$c_{\mathrm{gen}}(p)\leq1/3$ prove the schedule-dependent estimate. Hence
assume $\sigma>0$.

Let $y$ and the independent variables $\xi_1,\ldots,\xi_T$ be those from
\Cref{lem:jump_stochastic_lower_bound_standard_smoothness}. Choose
\[
    h^2\eqdef
    2^{-8}\min\cb{L\Delta_1,\frac{Ly^2}{S_T}} .
\]
Use the function $F$ from \Cref{lem:long-ramp} with $R=S_T$ and the given
step-sizes, start at $x_1=0$, and define the stochastic gradients by
\[
    \nabla f(x,\xi_t)=
    \begin{cases}
        \xi_t, & x\leq0,\\
        \nabla F(x), & x>0.
    \end{cases}
\]
Since $\nabla F(x)=0$ for all $x\leq0$ and each $\xi_t$ is centered, the
stochastic gradient is conditionally unbiased at every query point. Its
central $p$-moment is bounded by $\sigma^p$ on $(-\infty,0]$ by
\Cref{lem:jump_stochastic_lower_bound_standard_smoothness}, and is zero on
$(0,\infty)$. Also, \Cref{lem:long-ramp} gives
$F(x_1)-\inf_{x\in\R}F(x)=0\leq\Delta_1$.

On the event from \Cref{lem:jump_stochastic_lower_bound_standard_smoothness},
let $\tau\leq m$ be the first nonzero sample among $\xi_1,\ldots,\xi_m$. On
this event, $\xi_t=0$ for every $t<\tau$ and
$\eta_\tau(-\xi_\tau)=y$. Since $x_1=0$ and the oracle returns zero before
time $\tau$, the iterates stay at $0$ up to time $\tau$, and therefore
\[
    x_{\tau+1}
    =
    x_\tau-\eta_\tau\xi_\tau
    =
    y .
\]
The event has probability at least $\delta$ by
\Cref{lem:jump_stochastic_lower_bound_standard_smoothness}. From time
$\tau+1$ onward the iterates are positive, so the stochastic gradient is
deterministic and equal to the true gradient. \Cref{lem:long-ramp} therefore
yields
\[
    \pare{\frac1T\sum_{t=1}^T\norm{\nabla F(x_t)}}^2
    \geq
    \frac{h^2}{4}
    =
    2^{-10}\min\cb{L\Delta_1,\frac{Ly^2}{S_T}} .
\]
Since $y^2=\sigma^2A_{p,m}^2(4\delta)^{-2/p}$, this gives
\[
    \pare{\frac1T\sum_{t=1}^T\norm{\nabla F(x_t)}}^2
    \geq
    c_{\mathrm{gen}}(p)
    \min\cb{
        L\Delta_1,\,
        \frac{L\sigma^2A_{p,m}^2}{S_T}\delta^{-\frac2p}
    }.
\]

Combining this stochastic hard instance with the deterministic hard instance
from \Cref{prop:deterministic_lower_bound_standard_smoothness}, and choosing the
worse of the two, proves the schedule-dependent estimate, because
$c_{\mathrm{gen}}(p)\leq1/3$ and
\begin{align*}
&\max\cb{
        \frac13\min\cb{L\Delta_1,\Delta_1/S_T},\,
        c_{\mathrm{gen}}(p)\min\cb{L\Delta_1,
        L\sigma^2A_{p,m}^2S_T^{-1}\delta^{-2/p}}
    }\\
&\quad\geq
    c_{\mathrm{gen}}(p)
    \max\cb{
        \min\cb{L\Delta_1,\Delta_1/S_T},\,
        \min\cb{L\Delta_1,
        L\sigma^2A_{p,m}^2S_T^{-1}\delta^{-2/p}}
    }\\
&\quad=
    c_{\mathrm{gen}}(p)
    \min\cb{
        L\Delta_1,\,
        \max\cb{
            \frac{\Delta_1}{S_T},\,
            \frac{L\sigma^2A_{p,m}^2}{S_T}\delta^{-\frac2p}
        }
    } .
\end{align*}

It remains to derive the $T$-rate. Let $S_m\eqdef\sum_{t=1}^m\eta_t$. Since
the step-sizes are non-increasing and $m\geq2$,
\[
    S_T
    =
    S_m+\sum_{t=m+1}^T\eta_t
    \leq
    S_m+(T-m)\eta_{m+1}
    \leq
    S_m+\frac{T-m}{m}S_m
    \leq
    3S_m.
\]
The relation between the $\ell_1$- and $\ell_p$-norms then gives
\[
    A_{p,m}
    =
    \pare{\sum_{t=1}^m\eta_t^p}^{1/p}
    \geq
    \frac{S_m}{m^{1-\frac1p}}
    \geq
    \frac{S_T}{3T^{1-\frac1p}}.
\]
Consequently,
\[
    \max\cb{
        \frac{\Delta_1}{S_T},\,
        \frac{L\sigma^2A_{p,m}^2}{S_T}\delta^{-\frac2p}
    }
    \geq
    \max\cb{
        \frac{\Delta_1}{S_T},\,
        \frac{L\sigma^2S_T}{9T^{2-\frac2p}}\delta^{-\frac2p}
    }
    \geq
    \frac{\sigma\sqrt{L\Delta_1}}{3T^{1-\frac1p}}\delta^{-\frac1p}.
\]
Using this bound in the schedule-dependent estimate and $\min\cb{a,b/3}\geq\frac13\min\cb{a,b}$ proves
\eqref{eq:SGD_standard_smooth_lower_bound_optimized}, with the hidden constant
$c_{\mathrm{ni}}(p)\eqdef c_{\mathrm{gen}}(p)/3$.
\end{proof}

\paragraph{Discussion}
Comparing \eqref{eq:SGD_standard_smooth_average_gradient_rate} and
\eqref{eq:SGD_standard_smooth_lower_bound_optimized} in the noise-dominated
regime, we see that their stochastic terms have the same dependence on
$L$, $\Delta_1$, $\sigma$, $T$, and $p$.  They differ only in the power of
the inverse failure probability and in numerical constants: the upper bound
contains $\delta^{-2/p}$, whereas the lower bound contains $\delta^{-1/p}$.  In
particular, for any fixed $\delta$, the stochastic term in the upper bound of
\Cref{the:minibatch_SGD} cannot be improved by replacing its constant
step-size with any deterministic positive non-increasing step-size sequence.

This has two important implications.  First, when $p<2$, vanilla
\algname{SGD} is not sufficient to attain the minimax-optimal dependence on
${T}$ and $p$ under standard smoothness.  Its stochastic term decays as
$T^{-(p-1)/p}$, whereas clipping attains the optimal exponent
\citep{zhang2020adaptive}, and suitably tuned \algname{Normalized SGD} with
momentum attains the fully parameter-optimal rate
\citep{hubler2024gradient,liu2024nonconvex}
\begin{equation*}
    \Oc\pare{
        \frac{L\Delta_1}{T}
        +
        \pare{
            \frac{L\Delta_1\sigma^{\frac{p}{p-1}}}{T}
        }^{\frac{2(p-1)}{3p-2}}
    },
\end{equation*}
for the squared average-gradient criterion after $T$ iterations.  This method
uses one stochastic gradient oracle call per iteration, so its iteration and
oracle budgets coincide.

In the noise-dominated regime, the same improvement can be stated directly in
terms of the target accuracy and confidence.  Let $N_{\mathrm{SGD}}$ and
$N_{\mathrm{NSGD}}$ denote the corresponding total oracle complexities for
accuracy $\eps$ and failure probability $\delta$.  The lower bound above and
the high-probability guarantee for suitably tuned minibatch
\algname{Normalized SGD} \citep[Corollary~5]{hubler2024gradient} give, up to
constants depending only on $p$,
\begin{equation*}
    N_{\mathrm{SGD}}
    =\Omega\pare{
        \pare{
            \frac{\sigma\sqrt{L\Delta_1}}
            {\eps^2\delta^{1/p}}
        }^{\frac{p}{p-1}}
    },
    \quad
    N_{\mathrm{NSGD}}
    =\Oc\pare{
        \frac{
            L\Delta_1\sigma^{\frac{p}{p-1}}
        }{
            \eps^{2+\frac{p}{p-1}}
        }
        \operatorname{polylog}\pare{\frac1\delta}
    }.
\end{equation*}
The lower bound for vanilla \algname{SGD}
has exponent $2p/(p-1)$ in $1/\eps$, whereas the upper bound for minibatch
\algname{Normalized SGD} has exponent $2+p/(p-1)$.  The latter exponent is
strictly smaller when $p<2$, and the two coincide when $p=2$.
Thus, adaptive gradient-control mechanisms such as normalization and related
clipping methods can attain a strictly better dependence on $\eps$ and $p$
than vanilla \algname{SGD} when the noise has infinite variance.

Second, the same comparison exposes a fundamental difference in confidence
dependence.  Vanilla \algname{SGD} necessarily incurs the polynomial factor
$\delta^{-1/(p-1)}$, whereas minibatch \algname{Normalized SGD} depends only
polylogarithmically on $1/\delta$.  As in the convex setting discussed in
\Cref{subsec:convex_high_probability_lower_bound}, adaptive gradient control
is therefore essential for obtaining high-probability convergence with a
mild dependence on the failure probability.

\section{Experiments}\label{sec:experiments}
We consider a constrained convex (or strongly convex) optimization problem of the form
$$
    \min_{x \in \mathcal{X}} F(x) := \|Ax - b\|_1 + \frac{\mu}{2} \|x\|_2^2, \qquad \mathcal{X} := \{x \in \mathbb{R}^d : \|x\|_\infty \leq R\} ,
$$
where $A \in \mathbb{R}^{d \times d}$, $b \in \mathbb{R}^d$, $\mu \geq 0$ is a regularization parameter, $R > 0$ is a fixed radius of the $\ell_{\infty}$ ball. To simulate heavy-tailed noise, we augment the (sub)gradient of $F(\cdot)$ with synthetic noise drawn from a two-sided Pareto distribution. Specifically, we define the stochastic gradient oracle as
$$
    \nabla f(x_t, \xi_t) := A^\top \operatorname{sign}(Ax_t - b) + \mu x_t + \xi_t .
$$
where $\xi_t \in \mathbb{R}^d$ is an i.i.d. heavy-tailed noise vector generated as
$$
    (\xi_t)_i = s_i \cdot u_i^{-\frac{1}{\alpha}} , \quad s_i \sim \mathrm{Unif}(\{-1,1\}), \quad u_i \sim \mathrm{Unif}(0,1),
$$
with a tail index parameter $\alpha \in (1, 2].$ Notice that the above two-sided Pareto distribution has all moments $p \in (1,  \alpha)$ finite, while all moments larger or equal to $\alpha$ are infinite. This distribution is chosen to simulate the noise with infinite variance satisfying Assumptions~\ref{assum:pBM} and \ref{assum:pBCM}.

\paragraph{Parameters and evaluation.}
We fix the problem dimension to $d = 10$, generate matrix $A \sim \mathcal{N}(0,1)^{d \times d}$ and vector $b \sim \mathcal{N}(0,1)^d$ once for all experiments, and set the initial point to $x_0 = 100 \cdot \mathbf{1}_d$. The feasible region radius is set to $R = 10$, and experiments are run for $T = 1000$ iterations. The performance is evaluated based on $200$ independent runs of each algorithm reporting the average and one standard deviation across iterations.

\paragraph{Experiment 1. Sensitivity of convergence rate to step-size choice.} \begin{figure}[htbp]
	\centering
	\subfigure[$\alpha = 1.6$]{
		\includegraphics[width=0.3\textwidth]{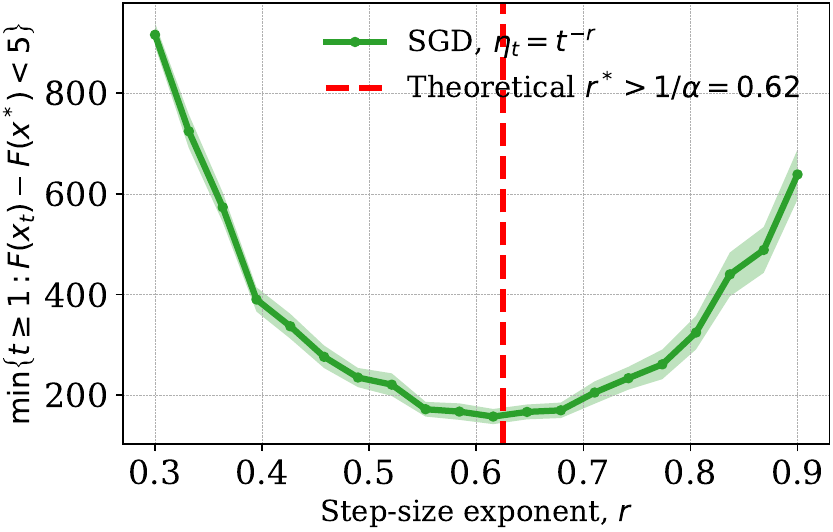}
		\label{fig:sensitivity_1_6}
	}
	\hfill
	\subfigure[$\alpha = 1.8$]{
		\includegraphics[width=0.3\textwidth]{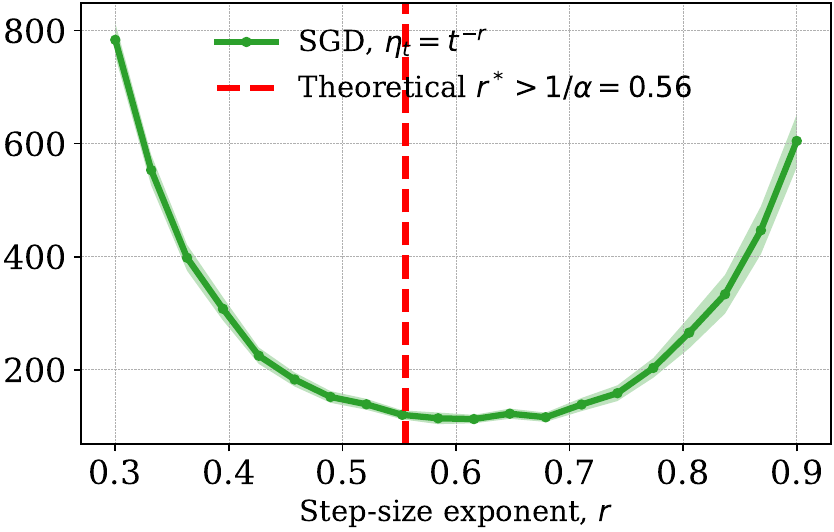}
		\label{fig:sensitivity_1_8}
	}
	\hfill
	\subfigure[$\alpha = 2.0$]{
		\includegraphics[width=0.3\textwidth]{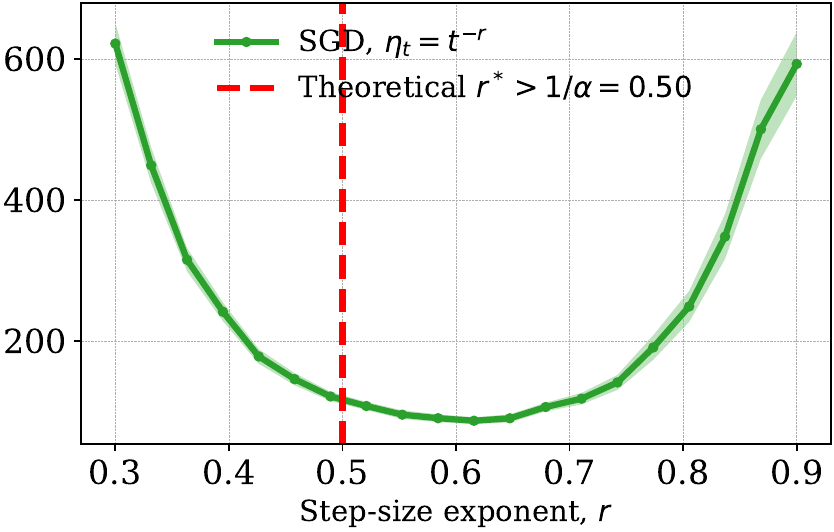}
		\label{fig:sensitivity_2_0}
	}
	\caption{Sensitivity of convergence rate to the step-size power $r$ in $\eta_t = 1/t^r$ for different values of the heavy-tail index $\alpha$. The minimal theoretical value for the optimal power $r = 1/\alpha$ is highlighted in each plot in red, and we can see that this value is often close the experimentally determined best value. }
	\label{fig:sensitivity_three}
\end{figure}
In our \Cref{thm:convex_SGD} for convex and \Cref{thm:NC_SGD_upper} for non-convex cases, the guarantee is established for any non-negative step-size sequence. However, the reccomended step-size order (ignoring parameters $G$, $D_{\cX}$, $\ell$ and $L$) is $\eta_t = 1/\sqrt[p]{t}$. In this experiment we aim to study the predictive power of our theory in a numerical experiment by varying different orders of the step-size. We set $\mu = 0$ and compare SGD with step-size $\eta_t = 1/t^r$ for $20$ different values of $r \in [0.3, 0.9]$. The sensitivity plot on \Cref{fig:sensitivity_three} shows how many iterations/samples $T$ it takes to reach the accuracy level $F(x_T) - F(x^*) \leq 5.$ vs. the step-size power $r.$ The red dashed line indicates the minimal theoretically reccomended value $ r = 1/\alpha.$ All three plots show that SGD is not very sensitive to the choice of step-size order $r$, and the value of $r = 1/\alpha$ is often close the experimentally determined best value (with lowest possible number of iterations).

\paragraph{Experiment 2. Effect of output selection in \algname{SGD} under heavy-tailed noise.}
\begin{figure}[htbp]
	\centering
	\subfigure[$\alpha = 1.2$]{
		\includegraphics[width=0.3\textwidth]{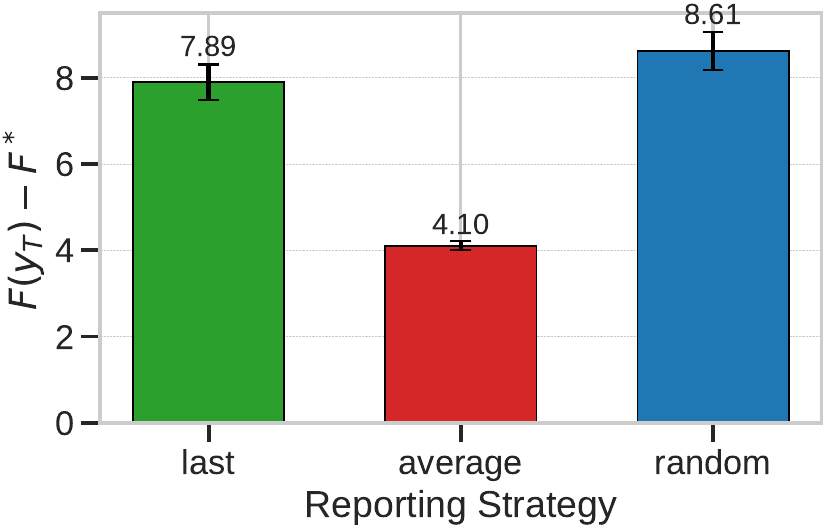}
		\label{fig:reporting_strategies_1_2}
	}
	\hfill
	\subfigure[$\alpha = 1.6$]{
		\includegraphics[width=0.3\textwidth]{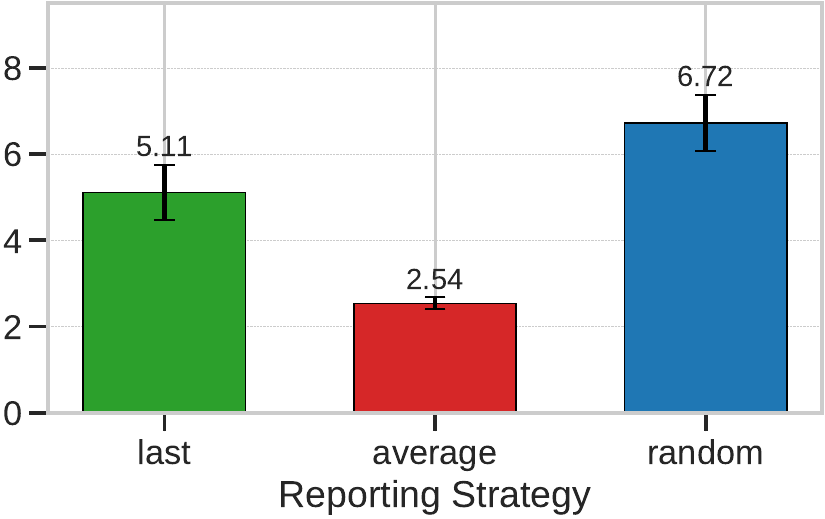}
		\label{fig:reporting_strategies_1_6}
	}
	\hfill
	\subfigure[$\alpha = 2.0$]{
		\includegraphics[width=0.3\textwidth]{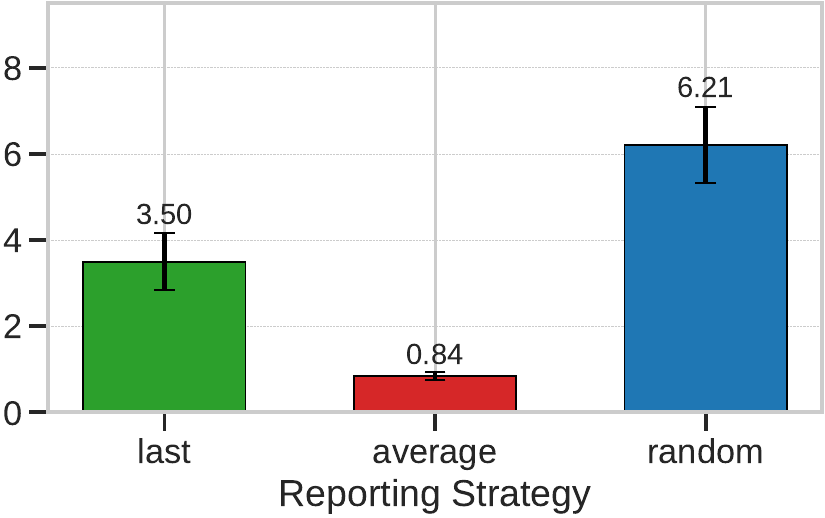}
		\label{fig:reporting_strategies_2_0}
	}
	\caption{Comparison of output selection strategies in stochastic gradient descent (SGD) under convexity ($\mu = 0$) and heavy-tailed noise. We evaluate the final suboptimality $F(y_T) - F^*$, where $y_T$ is the reported output, using three strategies: the last iterate, the average of all iterates, and a uniformly sampled iterate. Averaging the iterates yields the lowest suboptimality, significantly outperforming both the last iterate and random selection.
	}
	\label{fig:reporting_strategies}
\end{figure}
As we have seen, depending on convexity assumptions, our theory suggests different output strategies for \algname{SGD}. It is known that in convex setting under bounded variance assumption $p=2$, the last iterate converges with optimal complexity \citep{zamani2023exact}, however, our \Cref{thm:convex_SGD} requires to output the average iterate. In strongly convex setting, to obtain the optimal convergence in function value, it is typical to output the average iterate \citep{stich2019unified}, however, interestingly our \Cref{thm:SCprojectedSGD_improved} requires randomly sampled output for $p < 2$. To investigate the impact of these different strategies, we compare the function suboptimality of each for different noise levels. In this experiment, the step-size is set to $\eta_t = 1/\sqrt{t}$, as motivated by standard theory in the convex setting. After $T = 1000$ iterations, we compare three strategies for producing the final output $y_T$:
(i) the last iterate $x_T$,
(ii) the average $\widetilde{x}_T = \frac{1}{T} \sum_{t=1}^T x_t$, and
(iii) a randomly sampled iterate $\bar{x}_T  \sim \mathrm{Uniform}\{x_1, \ldots, x_T\}$.

For each strategy, we measure the expected suboptimality $F(y_T) - F^*$, where $F^*$ is the optimal objective value computed deterministically without noise. %
As shown in Figure~\ref{fig:reporting_strategies}, the average iterate significantly outperforms the other two strategies. This highlights the stabilizing effect of averaging in the presence of heavy-tailed stochastic gradients and is in line with theoretical guarantee in \Cref{thm:convex_SGD}.

\paragraph{Experiment 3. Comparison to \algname{Clip-SGD} in strongly convex setting.}
\begin{figure}[htbp]
	\centering
	\subfigure[$\alpha = 1.2$]{
		\includegraphics[width=0.3\textwidth]{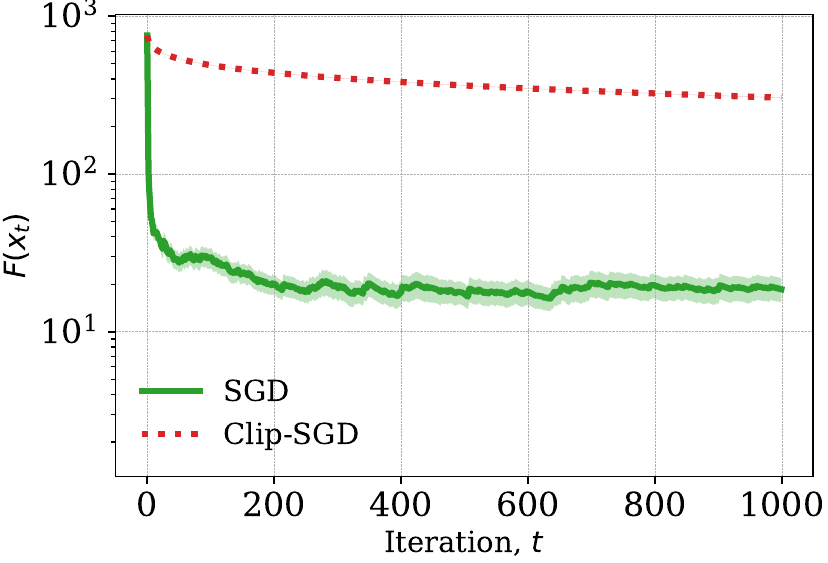}
		\label{fig:SC_results_1_2}
	}
	\hfill
	\subfigure[$\alpha = 1.6$]{
		\includegraphics[width=0.3\textwidth]{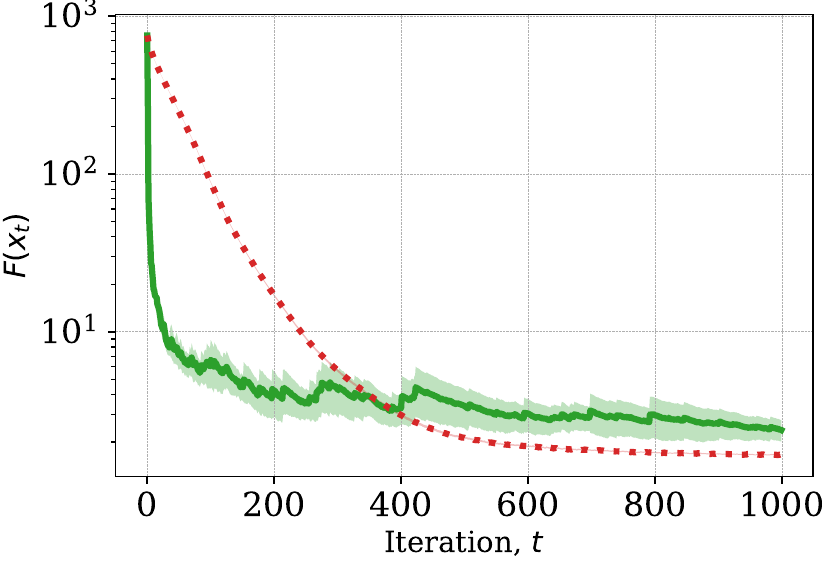}
		\label{fig:SC_results_1_6}
	}
	\hfill
	\subfigure[$\alpha = 2.0$]{
		\includegraphics[width=0.3\textwidth]{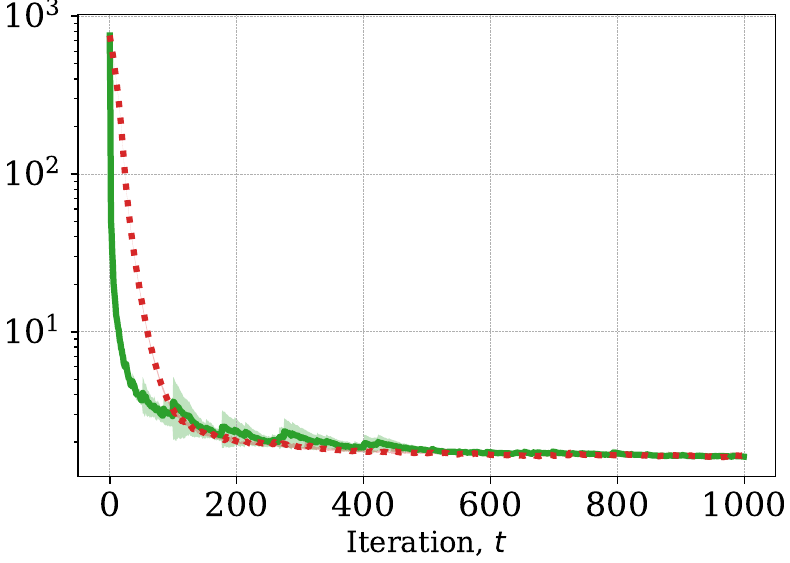}
		\label{fig:SC_results_2_0}
	}
	\caption{
		Comparison of \algname{SGD} and \algname{Clip-SGD} in the strongly convex setting for different values of the heavy-tail index $\alpha \in \{1.2, 1.6, 2.0\}$. Both algorithms use the same diminishing step-size schedule $\eta_t = 1/(\mu t)$, where $\mu > 0$ is the strong convexity parameter. For \algname{Clip-SGD}, the gradient clipping threshold is set as $\lambda_t = t^{\alpha - 1}$ without tuning, following theoretical recommendations from \cite{zhang2020adaptive,sadiev2023high}.
		Results show that \algname{Clip-SGD} improves the overall convergence stability under heavy-tailed noise, especially when $\alpha$ is small (e.g., $\alpha = 1.2, 1.6$), where standard \algname{SGD} suffers from high variance. However, in the initial optimization phase, \algname{SGD} often outperforms its clipped variant thanks to larger update steps.
	}
	\label{fig:comparison_to_Clip_SGD}
\end{figure}
In this experiment, we study the impact of gradient clipping in the strongly convex setting under heavy-tailed noise. Specifically, we compare standard (projected) stochastic gradient descent (SGD) with its clipped variant (\algname{Clip-SGD}). The algorithms are configured with the same step-size $\eta_t = 1 / (\mu t) $ as follows:
	\[
	\text{\algname{SGD}: } \qquad x_{t+1} = \Pi_{\mathcal{X}} \left( x_t - \eta_t \nabla f(x_t, \xi_t) \right) ,
	\]
	\[
		\text{\algname{Clip-SGD}: } \qquad x_{t+1} = \Pi_{\mathcal{X}} \left( x_t - \eta_t g_t \right), \quad g_t = \operatorname{clip}\left( \nabla f(x_t, \xi_t), \lambda_t \right)
	\]
	with $\operatorname{clip}(v, \lambda) := v \cdot \min\left\{1, \lambda / \|v\|_2 \right\}$. In \algname{Clip-SGD}, the clipping thresholds are set as $\lambda_t = t^{\alpha - 1}$ based on the theoretical analysis in \cite{zhang2020adaptive, sadiev2023high}, with no tuning.

Figure~\ref{fig:comparison_to_Clip_SGD} presents the evolution of the mean and the standard deviation of the objective value $F(x_t)$ over iterations for each method. As expected, \algname{Clip-SGD} significantly reduces the variance of \algname{SGD} across all noise levels. In some situations, e.g., $\alpha = 1.6$, this allows  \algname{Clip-SGD} to outperform \algname{SGD} after $t \geq 400$ iterations, where vanilla \algname{SGD} suffers from large variance and slow convergence. On the other hand, the experiment suggest that \algname{SGD} is sometimes competitive and can even outperform \algname{Clip-SGD}. First, we observe that when $\alpha$ increases toward the light-tailed regime ($\alpha = 2.0$), the performance gap narrows, confirming that clipping is especially beneficial under heavy-tailed stochasticity. Second, perhaps most surprisingly, in the most heavy-tailed regime $\alpha = 1.2$ and in the early phase of medium regime $\alpha = 1.6$, \algname{SGD} can significantly outperform \algname{Clip-SGD}. This happens because  \algname{Clip-SGD} is initially making very small steps, perhaps due to untuned clipping sequence $\lambda_t = t^{\alpha - 1}$, while \algname{SGD} can make large steps, quickly converging to a certain noise level.

\section*{Conclusion}
We have presented a comprehensive analysis of vanilla \algname{SGD} under heavy-tailed noise, establishing sharp convergence guarantees across convex, strongly convex, and non-convex settings under minimal $p$-th moment assumptions. Our results show that \algname{SGD} remains effective---even without adaptivity---in regimes with unbounded variance, achieving minimax optimal rates in convex and strongly convex cases, and tight algorithm-specific bounds in the non-convex setting.

Several avenues for future research remain open. One promising direction lies in refining convergence guarantees in smooth settings to better exploit the problem structures in the differentiable setting. Another is extending the analysis of heavy-tailed \algname{SGD} to distributed, federated, or communication-constrained environments. Third, our study assumes access to unbiased stochastic gradients; understanding the interplay between heavy tails and biased or corrupted gradient information arising, for example, from compression or privacy constraints merits further investigation.

\section*{Acknowledgments}
This work is supported by ETH AI Center Doctoral Fellowship, Swiss National
Science Foundation (SNSF) Project Funding No. 200021-207343, and the Office of Naval Research grant N00014-24-1-2654.

\appendix

\section{Useful Lemma}\label{appendix:useful_lemma}

First, we recall the definition of FBE of order $\nu$ that was given in \eqref{eq:FBE_order_nu}:
$$
\mathcal{D}_{\rho}^{\nu}(x) := - \frac{\nu \, \rho^{\frac{1}{\nu-1}} }{\nu - 1} \min_{y \in \cX} Q_\rho^{\nu}(x, y), \qquad
Q_\rho^{\nu}(x, y) := \langle \nabla F(x), y - x \rangle + \frac{\rho}{\nu} \norm{y - x}^{\nu}.
$$
The following lemma is a modified version of Step I in the proof of the main theorem in \citep{fatkhullin2024taming}. This descent type lemma is useful in the proof of \Cref{thm:NC_SGD_upper}.
\begin{lemma}\label{le:descent_FBE}
	For any $x\in \cX$ and any $\rho_1 \geq \rho + L$ we have $F_{1/\rho}^{\nu}(x) \leq F(x) - \fr{\nu-1}{\nu \rho_1^{\frac{1}{\nu-1}}} \mathcal D_{\rho_1}^{\nu}(x).$
\end{lemma}
\begin{proof}
 Notice that for any $x\in \cX $, we have for any $x^+ \in \cX $
$$
F_{1/\rho}^{\nu}( x)  = F( \hat x) + \frac{\rho}{\nu} \norm{\hat x - x}^\nu  \leq F(x^+) + \frac{\rho}{\nu} \norm{x^+ - x }^{\nu},
$$
where $\hat x \in \argmin_{y \in \cX} F(y) + \frac \rho \nu \norm{y - x}^\nu$ as before. We use a mirror descent step starting from a point $x$ as $x^+ := \argmin_{y \in \cX} \, \langle \nabla F(x), y \rangle + \frac{\rho_1}{\nu} \norm{y - x}^{\nu} $ with $\rho_1 \geq \rho + L$. Then by Hölder smoothness  of $F(\cdot)$ (upper bound in \Cref{ass:smooth})
\begin{eqnarray}%
    F_{1/\rho}^{\nu}( x) &\leq& F\left({x}^+\right) + \frac{\rho}{\nu} \norm{x^+ - x}^{\nu} \notag \\
& \leq& F\left(x\right)+\left\langle\nabla F\left(x\right), {x}^+ - x\right\rangle + \frac{L}{\nu} \norm{x^+ - x}^{\nu} + \frac{\rho}{\nu} \norm{x^+ - x}^{\nu}  \notag \\
& = & F\left(x\right) - \frac{\nu-1}{\nu \rho_1^{\frac{1}{\nu-1}}} \cD_{\rho_1}^{\nu}( x  )   + \frac{\rho + L - \rho_1}{\nu} \norm{x^+ - x}^{\nu} \leq F\left(x\right) - \frac{\nu-1}{\nu \rho_1^{\frac{1}{\nu-1}}} \cD_{\rho_1}^{\nu}( x  )  .
\end{eqnarray}
where the last equality holds by definitions of $x^+$, $ \cD_{\rho}^{\nu}( x  ) $ and the last step is due to condition $\rho_1 \geq \rho + L$.
\end{proof}

Recall the definition of the proximal point of order $\nu \in [1, 2]$
$$
\hat x^{\nu} := \argmin_{y\in \cX} \sb{ F(y) + \frac{\rho}{\nu} \norm{y - x}^{\nu}  } \quad \text{for any } x \in \cX .
$$
\begin{lemma}\label{le:PM_FBE_connection}
	Let $\nu \in (1, 2]$. For any $x\in \cX$ and any $\rho > \frac{\rho_1 + 2\,\ell}{\nu}$ we have $$
    \norm{\hat x^\nu - x}^\nu  \leq \frac{\nu-1}{\nu \rho_1^{\frac{1}{\nu-1}}} \cdot \frac{\mathcal D_{\rho_1}^{\nu}(x) }{\rho - \frac{\rho_1 + 2\,\ell}{\nu} }.
    $$
\end{lemma}
\begin{proof}
    By the definition of $\hat x^\nu$ and the optimality condition, we have for any $u\in \cX$
    $$
    \<\nabla F(\hat x^{\nu}) + \rho (\hat x^\nu - x) \norm{\hat x^\nu - x}^{\nu-2}, u - \hat x^\nu \> \geq 0 .
    $$
    Setting $u = x,$ we obtain
    \begin{equation}\label{eq:PM_split_I_II} \rho \norm{\hat x^\nu - x}^\nu \leq \langle\nabla F(\hat x^{\nu}), x - \hat x^\nu \rangle = \underbrace{\langle\nabla F(x), x - \hat x^\nu \rangle}_{\text{(I)}} + \underbrace{\langle\nabla F(\hat x^{\nu}) - \nabla F(x), x - \hat x^\nu \rangle}_{\text{(II)}} . \end{equation}
    We will bound terms (I) and (II) separately. First, by the definition of FBE:
    \begin{eqnarray*}
        \text{(I)} &=& \< \nabla F(x), x - \hat x^\nu \> - \frac{\rho_1}{\nu}\norm{\hat x^\nu - x}^\nu + \frac{\rho_1}{\nu}\norm{\hat x^\nu - x}^\nu \\
        &\leq& \max_{y\in\cX} \cb{ \< \nabla F(x), x - y \> - \frac{\rho_1}{p}\norm{y - x}^\nu } + \frac{\rho_1}{\nu}\norm{\hat x^\nu - x}^\nu \\
        &=& - \min_{y\in\cX} Q_{\rho_1}^\nu(x, y) + \frac{\rho_1}{\nu}\norm{\hat x^\nu - x}^\nu = \frac{\nu-1}{\nu \rho_1^{\frac{1}{\nu-1}}} \mathcal D_{\rho_1}^{\nu}(x) + \frac{\rho_1}{\nu}\norm{\hat x^\nu - x}^\nu .
    \end{eqnarray*}
    Second, by Hölder smoothness \Cref{ass:smooth} we have for any $x, y\in \cX$
    $$
	- \frac{\ell}{\nu} \norm{x-y}^{\nu} \leq F(x) - F(y) - \langle \nabla F(y), x - y \rangle ,
	$$
    $$
	- \frac{\ell}{\nu} \norm{x-y}^{\nu} \leq F(y) - F(x) + \langle \nabla F(x), x - y \rangle   .
	$$
    Summing up the above two inequalities for $x = x$ and $y = \hat x^{\nu}$, we can bound the term:
    $$
    \text{(II)} = - \langle\nabla F(\hat x^{\nu}) - \nabla F(x), \hat x^\nu  - x \rangle \leq \frac{2 \ell}{\nu} \norm{\hat x^{\nu} - x}^\nu .
    $$
    Using the two upper bounds in \eqref{eq:PM_split_I_II}, we derive
    $$
     \rho \norm{\hat x^\nu - x}^\nu \leq \frac{\nu-1}{\nu \rho_1^{\frac{1}{\nu-1}}} \mathcal D_{\rho_1}^{\nu}(x) + \frac{\rho_1 + 2\,\ell}{\nu}\norm{\hat x^\nu - x}^\nu .
    $$
    It remains to rearrange and use the restriction for $\rho.$
\end{proof}

\begin{lemma}[Lemma 10 in \cite{hubler2024gradient}]\label{lem:app.von_Bahr_and_Essen}
	Let $p \in [1,2]$, and $X_1, \ldots, X_n \in \R^d$ be a martingale difference sequence (MDS), i.e., $\Exp[X_{j-1}, \ldots, X_1]{X_{j}} = 0$ a.s. for all $j = 1, \ldots, n$ satisfying:
$
		\Exp{\norm{X_j}^p } < \infty \quad \text{for all } j = 1, \ldots, n.
$
	Define $S_n :=  \sum_{j=1}^n X_j$, then
$$
		\Exp{\norm{S_n}^p} \leq 2 \sum_{j=1}^n \Exp{\norm{X_j}^p}.
$$
\end{lemma}

{\color{black}{
\begin{lemma}[Theorem 6.3 in \cite{Rodomanov2020}]\label{lem:hoelder_smoothness_of_eucl_norm}
    Let $q \in \Nb_{\geq 1}, \nu \in [0,1]$ and $f \colon \R^d \to \R,\, f(x) \coloneqq \norm{x}^{q+\nu}$. Then the $(q-1)$-th derivative of $f$ is $(L, L, 1+\nu)$-Hölder smooth (c.f.\ Assumption \ref{ass:smooth}) with constant
    \begin{equation*}
        L = \begin{cases}
            (q-1)!! \prod_{i=1}^{\nicefrac q 2} (\nu + 2i) + \prod_{i=1}^q (\nu + 1), & \text{if $q$ is even},\\
            2^{1-\nu}\prod_{i=1}^q (\nu + i), & \text{if }q\text{ is odd.}
        \end{cases}
    \end{equation*}
\end{lemma}
}}

{\color{black}{
In particular, the above Lemma implies that $\norm{\cdot}^p$ is $(L', L', p)$-Hölder smooth with constant $L' = 2^{2-p}p$ by choosing $q = 1, \nu = p - 1$.
}}

\section{Missing Proofs}
\label{sec:app.missing_proofs}

\begin{proof}[Proof of \Cref{prop:nonconvex.sgd_lb.deterministic_small_steps}]
    This construction follows the idea from \citep{NSGDM_LzLo2023Huebler}, and adapts it {\color{black}{from \algname{Normalized SGD}}} to \algname{SGD} and {\color{black}{extends it from standard smoothness to}} Hölder-smoothness. Similar constructions can be found in \citep[Section 3.2]{PerformanceFirstOrderMethods2014Drori}.
    W.l.o.g.\ assume $x_1 = 0$ and define $y_t \coloneqq \sum_{\kappa = 1}^{t - 1} 2 \eps \eta_t$. Let
    \begin{equation*}
        T^* \coloneqq  \sup\set{T \in \Nb \mid 2 \eps y_T \leq \frac{\Delta_1}2}
    \end{equation*}
    and define $\delta_\nu \coloneqq \pare{\frac{2\eps}{L}}^{\frac 1 {\nu-1}}$,
    \begin{equation*}
        F(x) \coloneqq \begin{cases}
            \Delta_1 - 2 \eps x, & x \leq y_{T^*}\\
            \Delta_1 - 2 \eps x + \frac L {\nu} \pare{x - y_{T^*}}^{\nu}, & y_{T^*} < x \leq y_{T^*} + \delta_\nu \\
            \Delta_1 - 2 \eps y_{T^*} - 2 \eps \delta_\nu+ \frac L {\nu}\delta_\nu^{\nu} & y_{T^*} + \delta_\nu < x.
        \end{cases}
    \end{equation*}
    Note that, by the definition of $T^*$ and our assumption $\eps^{\frac{\nu} {\nu-1}} \leq \frac{\nu} {\nu-1} \frac{\Delta_1}{4}\pare{\frac{L}{2}}^{\frac 1 {\nu-1}}$, we have
    \begin{equation*}
        F(x)
        \geq F\pare{y_{T^*} + \delta_\nu}
        = \Delta_1 - 2 \eps y_{T^*} - 2 \eps \delta_\nu+ \frac L {\nu}\delta_\nu^{\nu}
        \geq \frac{\Delta_1}2 - 2 \eps \delta_\nu+ \frac L {\nu}\delta_\nu^{\nu}
        \geq 0
    \end{equation*}
    and hence $F(x_1) - \inf_{x \in \R} F(x) \leq \Delta_1$. Furthermore $F$ is $(0, L, \nu)$-Hölder smooth and convex by construction. Next, note that
    \begin{equation*}
        F'(x) \coloneqq \begin{cases}
            -2 \eps, & x \leq y_{T^*}\\
            - 2 \eps + L \pare{x - y_{T^*}}^{\nu - 1}, & y_{T^*} < x \leq y_{T^*} +\delta_\nu \\
            0 & y_{T^*} + \delta_\nu < x.
        \end{cases}
    \end{equation*}
    and hence we have $\norm{\nabla F(y_t)} = 2 \eps > \eps$ for all $t \leq T^*$. Finally note that \algname{SGD}, when started at $x_1 = 0$, observes $F'(x_t) = -2\eps$ as long as $x_t \leq y_{T^*}$. In particular, the iterates are given by $x_t = \sum_{\kappa = 1}^{t-1} 2 \eps \eta_t = y_t$, and hence $\abs{F'(x_t)} = 2\eps > \eps$ for all $t \leq T^*$.
\end{proof}

\begin{proof}[Proof of \Cref{thm:nonconvex.lower_bound.SGD_lower_bound}]
    We first note that the functions in \Cref{prop:nonconvex.sgd_lb.deterministic_small_steps} and \Cref{prop:nonconvex.sgd_lb.stoch_large_steps} are convex, $(0,L, \nu)$-smooth and $G$-Lipschitz by their definitions and our assumption $\nicefrac \eps 2 < G$. Additionally note that we can lift the function $F$ from \Cref{prop:nonconvex.sgd_lb.deterministic_small_steps} to $F_1 \colon \R^d \to \R$, by setting $F_1(x) = F(x_1)$. Hence we can use these constructions for $F_1, \nabla f_1$ and $F_2, \nabla f_2$ respectively.

    Now let us first assume $T < \pare{\frac{\pare{1-r}\Delta_1}{8 \eta \eps^2}}^{\frac 1 {1-r}}$. Then we have
    \begin{equation*}
        \sum_{t = 1}^{T-1} \eta_t \leq \eta \pare{1 + \int_1^{T-1} t^{-r}} = \eta \pare{1 + \frac{T^{1-r} -1}{1-r}} \leq \frac{\eta T^{1-r}}{1-r} < \frac{\Delta_1}{8 \eps^2}
    \end{equation*}
    and hence $\min_{t \in [T]}\norm{\nabla F_1(x_t)} > \eps$ by \Cref{prop:nonconvex.sgd_lb.deterministic_small_steps}.

    Next we choose $F_2, \nabla f_2$ from \Cref{prop:nonconvex.sgd_lb.stoch_large_steps} and $\norm{x_1} = \pare{\frac{2^{2-\nu} \nu \Delta_1}{L}}^{\nicefrac 1 \nu}$, which guarantees $F_2(x_1) - \inf_x F_2(x) \leq \Delta_1$. By iteratively applying \Cref{prop:nonconvex.sgd_lb.stoch_large_steps}, we get
    \begin{equation*}
        \min_{t \in [T]} \norm{x_t} \geq \min\set{\norm{x_1}, \min_{t \in [T]} \tau_t}, \qquad \text{where} \qquad \tau_t \coloneqq \frac 1 2 \pare{\frac{\eta_t^{p-1}\sigma^p}{2^pL}}^{\frac 1 {p + \nu - 2}}.
    \end{equation*}
    By our non-degeneration assumptions, i.e., $\eps < \frac G 2$ and $\eps^{\frac \nu {\nu - 1}} < \nu \Delta_1 \pare{\frac L {2^{2-\nu}}}^{\frac 1 {\nu - 1}},$ we have $\norm{\nabla F(x_1)} > \eps$ and hence no $\eps$-stationary point can be reached before the inequality $\tau_T = \min_{t \in [T]} \tau_t \leq \pare{\frac{2^{2-\nu}\eps}{L}}^{\frac{1}{\nu - 1}}$ is satisfied. Rewriting this inequality yields
    \begin{align*}
        \tau_T \leq \pare{\frac{2^{2-\nu}\eps} L}^{\frac 1 {\nu - 1}}
        &\Leftrightarrow \frac{\eta_T^{p-1} \sigma^p}{2^pL} \leq 2^{p + \nu - 2} \pare{\frac{2^{2-\nu}\eps} L}^{\frac {p + \nu - 2} {\nu - 1}}\\
        &\Leftrightarrow \eta T^{-r} \leq \pare{\frac {2^p L} {\sigma^p}}^{\frac{1}{p-1}} 2^{\frac{p+\nu -2}{(p-1)(\nu - 1)}} \pare{\frac \eps L}^{\frac{p + \nu - 2}{(p-1)(\nu - 1)}}.
    \end{align*}
    As this inequality is not satisfied whenever $T < \eta^{\frac 1 r} \pare{\frac{\sigma^p}{2^pL}}^{\frac 1 {r\pare{p-1}}} \pare{\frac{L}{2\eps}}^{\frac{p - 1 + \nu - 1}{r\pare{p-1}\pare{\nu-1}}}$, we get the claim.
\end{proof}

{\color{black}{The proof of the following proposition is similar to that of}}
\Cref{prop:nonconvex.sgd_lb.deterministic_small_steps}.

\begin{proof}[Proof of~\Cref{prop:deterministic_lower_bound_standard_smoothness}]
If $\Delta_1=0$, take $F\equiv0$.  Hence assume $\Delta_1>0$.
Set
\[
\eps^2=c_0\min\cb{L\Delta_1,\Delta_1/S_T}
\]
with $c_0=1/12$.
Start at $x_1=0$, set $Y=2\eps S_T$, and define
\[
\nabla F(x)=
\begin{cases}
    -2\eps, & x\leq Y,\\
    -2\eps+L(x-Y), & Y<x<Y+2\eps/L,\\
    0, & x\geq Y+2\eps/L .
\end{cases}
\]
Let $F(0)=\Delta_1/2$ and define $F$ by integration.  The gradient is
$L$-Lipschitz.  The function increases to the left of $0$ and becomes constant
to the right of $Y+2\eps/L$, so it is lower bounded.  Its only decrease from
$x=0$ occurs before the flat region, and this decrease is at most
\[
    2\eps Y+\frac{2\eps^2}{L}
    =
    4\eps^2S_T+\frac{2\eps^2}{L}
    \leq
    6c_0\Delta_1
    \leq
    \Delta_1/2.
\]
Thus $F(0)-\inf_{x\in\R}F(x)\leq\Delta_1$.
Along \algname{SGD} with the deterministic oracle,
$x_t=2\eps\sum_{i=1}^{t-1}\eta_i\leq Y$ for every $t=1,\ldots,T$. Therefore
$\nabla F(x_t)=-2\eps$ for all these iterates, and
\[
    \pare{\frac1T\sum_{t=1}^T\norm{\nabla F(x_t)}}^2
    =
    4\eps^2
    =
    \frac13\min\cb{L\Delta_1,\Delta_1/S_T} .
\]
\end{proof}

\bibliographystyle{alpha}
\bibliography{biblio}

\end{document}